\documentclass[11pt,letterpaper]{amsart}
\usepackage{amsmath}
\usepackage{amssymb}
\usepackage{amscd}
\usepackage{mathrsfs}
\usepackage{url}
\usepackage[all]{xy}
\usepackage[dvipdfmx]{graphicx}
\usepackage{graphics}
\usepackage{tikz}
\usepackage{ulem}
\theoremstyle{plain} %text of this environment is typesetted in italics
\newtheorem{theorem}{\indent\sc Theorem}[section]
\newtheorem{lemma}[theorem]{\indent\sc Lemma}
\newtheorem{corollary}[theorem]{\indent\sc Corollary}
\newtheorem{proposition}[theorem]{\indent\sc Proposition}

\theoremstyle{definition} %text of this environment is typesetted in roman letters
\newtheorem{definition}[theorem]{\indent\sc Definition}
\newtheorem{remark}[theorem]{\indent\sc Remark}
\newtheorem{example}[theorem]{\indent\sc Example}

\newtheorem{question}[theorem]{\indent\sc Question}

\newcommand{\del}{{\partial}}
\newcommand{\delbar}{\overline{\partial}}
\newcommand{\ddbar}{\partial\overline{\partial}}

\newcommand{\vp}{\varphi}

\begin{document}

\pagestyle{plain}
\thispagestyle{plain}

\title[Formal principle for line bundles]
{Formal principle for line bundles on neighborhoods of an analytic subset of a compact K\"ahler manifold}%on compact K\"ahler manifolds}

\author[Takayuki KOIKE]{Takayuki KOIKE$^{1}$}
\address{ % First Author
$^{1}$ Department of Mathematics \\
Graduate School of Science \\
Osaka Metropolitan University \\
3-3-138 Sugimoto \\
Osaka 558-8585 \\
Japan 
}
\email{tkoike@omu.ac.jp}

%%%%%%%%%%%%%%% footnote %%%%%%%%%%%%%%%%
\renewcommand{\thefootnote}{\fnsymbol{footnote}}
\footnote[0]{ %2010 MSC numbers
}
\footnote[0]{ %key words and phrases
}
\renewcommand{\thefootnote}{\arabic{footnote}}
%\footnote{ %acknowledgment of support etc. if any
%$^{*}$Thanks.
%}
%%%%%%%%%%%%%%%%%%%%%%%%%%%%%%%%%%%%%%%%%

%%%%%%%%%%%%%%%%%%%%%%%%%%%%%%%%%%%%%%%%%%
\begin{abstract}
We investigate the formal principle for holomorphic line bundles on neighborhoods of an analytic subset of a complex manifold 
mainly in the case where it can be realized as an open subset of a compact K\"ahler manifold. 
Our approach identifies the obstruction as a global analytic class supported on a neighborhood of $Y$, 
and relates its vanishing to the solvability of a $\partial\overline{\partial}$-problem on neighborhoods of $Y$. 
As a consequence we obtain cohomological criteria ensuring the formal principle.
We also construct a holomorphic family of compact K\"ahler surfaces containing a curve with topologically trivial normal bundle in which the formal principle holds for almost every fiber but fails for uncountably many fibers, exhibiting an instability phenomenon in families.
\end{abstract}
%%%%%%%%%%%%%%%%%%%%%%%%%%%%%%%%%%%%%%%%%%

\maketitle

%%%%%%%%%%%%%%%%%%%%%%%%%%%%%%%%%%%%%%%%%%

\section{Introduction}
Let $X$ be a complex manifold and let $Y\subset X$ be a reduced analytic subset. 
Our interest in this paper is in comparing two holomorphic line bundles $L$ and $L'$ defined on a neighborhood $V$ of $Y$ in $X$. 
More precisely, we are interested in obtaining a practical sufficient condition for the existence of an open neighborhood $V_0$ of $Y$ in $V$ such that the restrictions $L|_{V_0}$ and $L'|_{V_0}$ are isomorphic to each other as holomorphic line bundles. 
By replacing $L$ with $L\otimes (L')^{-1}$, without loss of generality we may assume that $L'$ is holomorphically trivial. 
Hence the problem is reduced to investigating a sufficient condition for $L$ to be holomorphically trivial on a neighborhood of $Y$ in $V$. 
In what follows, we denote by $\mathcal{O}_X$ the sheaf of germs of holomorphic functions on $X$, by $\mathcal{O}_V(L)$ the sheaf of germs of holomorphic sections of $L$ on $V$, and by $\mathcal{I}_Y\subset \mathcal{O}_X$ the defining ideal sheaf of $Y$. 

An obvious necessary condition for the triviality of $L$ on a neighborhood of $Y$ is the triviality of the invertible sheaf $i_\mu^*\mathcal{O}_V(L)$ on $Y_\mu$ for any non-negative integer $\mu$, where $Y_\mu$ is the complex space $(Y, \mathcal{O}_X/\mathcal{I}_Y^{\mu+1})$ ({\it the $\mu$-th infinitesimal neighborhood of $Y$}) and $i_\mu\colon Y_\mu\to V$ is the natural morphism. %(Note that $i_0$ is the inclusion of $Y=Y_0$). 
One can describe a more precise sufficient condition by using the notion of {\it the formal completion of $X$ along $Y$}, which is the formal complex space (in the sense of \cite{Bi}) $Y_\infty=(Y, \textstyle\varprojlim_\mu\mathcal{O}_X/\mathcal{I}_Y^{\mu+1})$. See \cite{Ba} and \cite[Chapter II \S 4]{GPR} for the details. Denote by $\mathcal{O}_{Y_\infty}$ the sheaf $\textstyle\varprojlim_\mu\mathcal{O}_X/\mathcal{I}_Y^{\mu+1}$ and by $i_\infty\colon Y_\infty\to V$ the natural morphism. 
Then it is clear that the pull-back $i_\infty^*\mathcal{O}_V(L)$ is trivial as an invertible sheaf on $Y_\infty$ if $L$ is trivial on a neighborhood of $Y$. 
Note that, from \cite[Lemme 2.3 $(ii)$]{Ba}, it follows that $i_\infty^*\mathcal{O}_V(L)$ is trivial as an invertible sheaf on $Y_\infty$ if and only if there exists an element of $\widetilde{\sigma}\in \textstyle\varprojlim_\mu \Gamma(V, \mathcal{O}_V(L)\otimes \mathcal{O}_V/\mathcal{I}_Y^\mu)$ whose image by the natural map 
\begin{equation}\label{eq:naturalmapfromvarprojlimLrestr}
\varprojlim_\mu \Gamma(V, \mathcal{O}_V(L)\otimes \mathcal{O}_V/\mathcal{I}_Y^\mu) \to \Gamma(V, \mathcal{O}_V(L)\otimes \mathcal{O}_V/\mathcal{I}_Y)
\end{equation}
is a global nowhere vanishing holomorphic section of $L|_Y=i_0^*L$ on $Y$ (See Lemma \ref{lem:rel_probs}), where $\Gamma(V, -)$ denotes the global sections functor on $V$. 
It seems natural to ask whether the converse is true: i.e., whether the formal principle for (holomorphic) line bundles holds on neighborhoods of $Y$ in $X$ in the following sense. 
\begin{definition}
Let $X$ and $Y$ be as above. 
We say that {\it the formal principle for line bundles} ({\it the FPLB} for short) {\it holds on neighborhoods of $Y$ in $X$} if, for any neighborhood $V$ of $Y$ in $X$ and any holomorphic line bundle $L$ on $V$, there exists a neighborhood $V_0$ of $Y$ in $V$ such that the restriction $L|_{V_0}$ is holomorphically trivial if $i_\infty^*\mathcal{O}_V(L)$ is trivial (as an invertible sheaf on $Y_\infty$, namely $i_\infty^*\mathcal{O}_V(L)$ is isomorphic to $\mathcal{O}_{Y_\infty}$ as a sheaf of $\mathcal{O}_{Y_\infty}$-modules). 
\end{definition}
Note that we will investigate the FPLB for a given pair $(Y, X)$ (i.e., by fixing the complex structure of a neighborhood of $Y$) in the present paper, whereas the formal principle for the complex structures themselves of a neighborhood of analytic subspaces are also important and studied 
by several authors \cite{G} \cite{HR} \cite{Gri} \cite{A} \cite{Hi} \cite{ABT} \cite{PT} \cite{Hw}  \cite{K2022} \cite{GS1} \cite{GS2} \cite{GS3}, e.t.c., see also \cite[Chapter VII \S 4]{GPR}. 
For a given datum on the formal completion $Y_\infty$, the extendability and the uniqueness problems of an analytic object such as a holomorphic line bundle is systematically investigated by Griffiths in \cite{Gri}. 
Following his terminology, our problem can be rephrased as {\it the formal uniqueness problem} of a holomorphic line bundle from $Y$, or equivalently as {\it the formal extension problem} of a global nowhere vanishing holomorphic section of $L|_Y$ for each holomorphic line bundle $L$ on a neighborhood of $Y$, in view of the morphism $(\ref{eq:naturalmapfromvarprojlimLrestr})$. 

In accordance with Griffiths' arguments, one can deduce from \cite[Theorem III]{Gri} that the FPLB holds on neighborhoods of $Y$ in $X$ for example when $Y$ is a submanifold of ${\rm dim}\,Y\geq 2$ whose normal bundle is positive in the sense of Griffiths, where ${\rm dim}\,Y$ denotes the (complex) dimension of $Y$ (see Theorem \ref{thm:griffmain}). 
According to a theorem of Peternell \cite{P}, the FPLB holds on neighborhoods of $Y$ in $X$ when $Y$ is an exceptional subset in the sense of Grauert \cite{G}. 
When $Y$ is a submanifold and its normal bundle $N_{Y/X}$ is topologically trivial, it is known that the issue is quite delicate. 
In \cite[\S 5.4]{U}, Ueda explicitly constructed an example of a compact non-singular curve $Y$ embedded  in a complex surface $X$ with topologically trivial normal bundle such that the FPLB {\it does not} hold on neighborhoods of $Y$ in $X$ (See also \cite{A}). 
After \cite{U}, on a neighborhood of a submanifold with topologically trivial normal bundle, the FPLB-type problem has been investigated mainly for a special holomorphic line bundle, which is called {\it the Ueda line bundle} in \cite{PT}, to investigate {\it the vertical linearizability problem} in the sense of \cite{GS1} by applying such dynamical techniques as Siegel-type methods \cite{S} and Newton methods \cite{A} to construct an element of $\widetilde{\sigma}\in \textstyle\varprojlim_\mu \Gamma(V, \mathcal{O}_V(L)\otimes \mathcal{O}_V/\mathcal{I}_Y^\mu)$ % whose image by the map $(\ref{eq:naturalmapfromvarprojlimLrestr})$ is a global nowhere vanishing holomorphic section of $L|_Y=i_0^*L$ on $Y$ 
as above under some irrational number theoretical assumptions on $N_{Y/X}$ \cite{K2020} \cite{GS1} \cite{O}. See also \cite{K2017}, \cite{K2022}, and \cite{KU} for some generalizations of such techniques for the case where $Y$ is singular. 
%\cite{K2021}や\cite{K2024}ではこの問題\ref{prob:main}を, $L$のsemi-positivityの下で考察していている. ※\cite{K2021}では応用としてsemi-positivity判定もしている (自分の3本目のでも)

In the present paper we investigate criteria for the FPLB to hold on neighborhoods of $Y$ in $X$ mainly when $X$ can be realized as an open subset of a compact K\"ahler manifold. 
Note that we may assume $X$ itself is a compact K\"ahler manifold in this case without loss of generality. 
As a main result, we show the following: 
\begin{theorem}\label{thm:main}
Let $X$ be a compact K\"ahler manifold and $Y$ be a non-empty analytic subset of $X$. Denote by $M$ the complement $X\setminus Y$. Then the following hold: \\
$(i)$ Assume that the natural map $H_c^{2}(M, \mathcal{O}_M)\to H^{2}(M, \mathcal{O}_M)$ from the second cohomology group with compact support to the second cohomology group is injective. 
Then, for any holomorphic line bundle $L$ on a neighborhood $V$ of $Y$ in $X$ whose restriction $L|_Y (=i_0^*L)$ is holomorphically trivial, there exists an open neighborhood $V_0$ of $Y$ in $V$ such that the restriction $L|_{V_0}$ is holomorphically trivial. 
In particular, the FPLB holds on neighborhoods of $Y$ in $X$ in this case. \\
$(ii)$ Assume that there exists an effective divisor $D$ on $X$ such that ${\rm Supp}\,D=Y$ holds for the support ${\rm Supp}\,D$ of $D$ and that the natural map $H_c^{2}(M, \mathcal{O}_M)\to \textstyle\varprojlim_\mu H^2(X, \mathcal{O}_X(-\mu D))$ is injective. Then the FPLB holds on neighborhoods of $Y$ in $X$. 
\end{theorem}

By virtue of Ohsawa's theorem on the cohomology groups of K\"ahler manifolds which are {\it very strongly $q$-convex} \cite {O1982} (See also \cite{D1}, in which the same property is referred to as {\it absolutely $q$-convex}), from Theorem \ref{thm:main} $(i)$ one can deduce the following: 
\begin{corollary}\label{cor:main1}
Let $X$ be a compact K\"ahler manifold of dimension $n \geq 3$ and $Y$ be a non-empty analytic subset of $X$. 
Assume that the complement 
$M:=X\setminus Y$ is very strongly $(n-2)$-convex: i.e., there exists a $C^\infty$ exhaustion plurisubharmonic function on $M$ whose Levi form has at least $3$ positive eigenvalues at any point of $V\setminus Y$ for some neighborhood $V$ of $Y$. 
Then, for any holomorphic line bundle $L$ on a neighborhood $V$ of $Y$ in $X$ whose restriction $L|_Y$ is holomorphically trivial, there exists an open neighborhood $V_0$ of $Y$ in $V$ such that the restriction $L|_{V_0}$ is holomorphically trivial. 
%Then the FPLB holds on a neighborhood $Y$ of $X$. 
\end{corollary}

%See \S \ref{section:exdiss} for a comparison between Corollary \ref{cor:main1} and conclusions from some results in \cite{Gri} mainly when $Y$ is a submanifold. 
We will see in \S \ref{section:exdiss} that one can deduce a variant of another theorem of Ohsawa \cite{O2007} from Corollary \ref{cor:main1} (Proposition \ref{prop:variantofo2007}). 

One of the main features of Theorem \ref{thm:main} $(ii)$ is its applicability to the case where $X$ is a surface and $Y$ is a curve. Indeed, from this one can deduce the following: 
\begin{corollary}\label{cor:main2}
Let $X$ be a compact K\"ahler surface and $D$ ($\not=0$) be an effective divisor on $X$. Denote by $Y$ the support of $D$ and by $M$ the complement $X\setminus Y$. Assume that either of the following two conditions holds:\\
$(i)$ $H_c^2(M, \mathcal{O}_M)$ is Hausdorff (see \S \ref{subsection:prelimcoh} for the topology on this space) and the image of the natural map $\textstyle\varinjlim_\mu \Gamma(X, \mathcal{O}_X(K_X)\otimes \mathcal{O}_X(\mu D))\to \Gamma(M, \mathcal{O}_M(K_M))$ is dense in the topology of uniform convergence on compact subsets, where $K_M$ denotes the canonical line bundle $\textstyle\bigwedge^2T_M^*$ of $M$, or \\
$(ii)$ There exists a $C^\infty$ Hermitian metric on $[D]$ whose Chern curvature form is positive definite at any point of $Y$, where $[D]$ is the holomorphic line bundle on $X$ associated to the divisor $D$. \\
Then the FPLB holds on neighborhoods of $Y$ in $X$. 
\end{corollary}

Note that, for a compact K\"ahler surface $X$ and a connected non-singular curve $Y\subset X$, the divisor $D=Y$ satisfies the condition $(ii)$ in Corollary \ref{cor:main2} if the degree ${\rm deg}\,N_{Y/X}$ of the normal bundle $N_{Y/X}$ of $Y$ is positive (See Remark \ref{rmk:hypersurfaceYstYY>0}). 
Therefore it follows from Corollary \ref{cor:main2} that the FPLB holds on neighborhoods of $Y$ in $X$ in this case. 
When ${\rm deg}\,N_{Y/X}<0$, one can apply a theorem of Peternell which we mentioned above to deduce that the FPLB holds on neighborhoods of $Y$ in $X$, because $Y$ is exceptional in this case \cite{G}. 
For the case where ${\rm deg}\,N_{Y/X}=0$, we show the following: 
\begin{theorem}\label{thm:exp2bl9pts}
There exist a proper holomorphic submersion $\pi\colon \mathcal{X}\to R$ from a complex manifold $\mathcal{X}$ of dimension $3$ onto a compact Riemann surface $R$ and a compact submanifold $\mathcal{Y}$ of $\mathcal{X}$ which satisfy the following conditions: \\
$(i)$ For any $t\in R$, $X_t:=\pi^{-1}(t)$ is a projective manifold of dimension $2$ and $Y_t:=X_t\cap \mathcal{Y}$ is a submanifold of $X_t$ whose normal bundle is topologically trivial. \\
$(ii)$ For almost every $t\in R$ (in the sense of Lebesgue measure), the FPLB holds on neighborhoods of $Y_t$ in $X_t$. \\
$(iii)$ There exist uncountably many points $t\in R$ such that the FPLB does not hold on neighborhoods of $Y_t$ in $X_t$. 
\end{theorem}

In \S \ref{subsection:curvsurf}, we investigate the FPLB problem on neighborhoods of an elliptic curve in the projective plane blown up at nine general points of a smooth cubic, and give a characterization in terms of complex analytical and geometrical properties of the complement of $Y$ and the anticanonical line bundle 
(Proposition \ref{prop:cor3}). 
Theorem \ref{thm:exp2bl9pts} will be shown by applying this characterization and our previous results in \cite{K20242} on this example. 
See also \cite[\S 3]{KUprojK3} and \cite[Lemma 4.6]{K20242} for the FPLB problem on neighborhoods of a compact non-singular curve with topologically trivial normal bundle. 

We prove Theorem \ref{thm:main} by showing the unitary flatness of a given line bundle $L$ on a neighborhood of $Y$ under the assumptions as in the theorem. 
As $L|_Y$ is holomorphically trivial, it follows from a simple topological argument that there exists an open neighborhood $V_0$ of $Y$ such that $L|_{V_0}$ is topologically trivial (by taking $V_0$ such that $Y$ is a deformation retract of $V_0$, see \S \ref{subsection:prelimcoh}). 
Therefore, in view of an argument for proving the unitary flatness of a topologically trivial holomorphic line bundle on a complex manifold on which the $\ddbar$-lemma holds in the sense of \cite{KT} (see \S \ref{subsection:kasiwarathm}), the problem is reduced to solving the $\ddbar$-problem on a neighborhood of the boundary of the complement $M$ of $Y$. We will define an obstruction class $v(\mathcal{I}_Y, L)\in {\rm Ker}(H_c^2(M, \mathcal{O}_M)\to H^2(M, \mathcal{O}_M))$ to the solvability of this problem (\S \ref{section:def_obstr_class_vL}) and show the triviality of $L$ on a neighborhood of $Y$ when it vanishes (Proposition \ref{prop:van_v_implies_Ltriv}). 

The organization of the paper is as follows. 
\S 2 is a preliminary section. Here we prepare some notation and fundamental tools which are used throughout this paper. 
Based on \cite{Gri}, we make some formal (purely sheaf-theoretic) observations on the obstructions and give some fundamental sufficient conditions for the triviality of a given holomorphic line bundle on a neighborhood of $Y$ in \S 3. 
In \S 4.1, we first give a proof of (a part of) Theorem \ref{thm:main} $(i)$ under an additional assumption 
in order to make our strategy in the proof of Theorem \ref{thm:main} clear. 
Based on an observation in this section, we give the definition of the class $v(\mathcal{I}_Y, L)$ in \S 4.2. 
The proof of Theorem \ref{thm:main} and its corollaries will be given in \S 5. 
In \S 6, we will make further discussions by comparing our results with some known results. 
We also give some examples and prove Theorem \ref{thm:exp2bl9pts} in this section. 
%give an example and make further discussions. 
%
\vskip3mm
{\bf Acknowledgment. } 
The author would like to thank Prof. Takato Uehara, Prof. Yohsuke Matsuzawa, Dr. Satoshi Ogawa, and Dr. Jinichiro Tanaka for discussions. 
He is supported by the Grant-in-Aid for Scientific Research C (23K03119) from JSPS. 
%%%%%%%%%%%%%%%%%%%%%%%%%%%%%%%%%%%%%%%%%%

%%%%%%%%%%%%%%%%%%%%%%%%%%%%%%%%%%%%%%%%%%
\section{Preliminaries}\label{section:prelim}

\subsection{Notation}

In the present paper, any complex manifold will be assumed to be paracompact and Hausdorff. 
To denote a complex manifold, in what follows we choose to use the letter $\Omega$ when the compactness assumption is not needed, and keep the letter $X$ for compact manifolds. 

For a complex manifold $\Omega$ of dimension $n$ and a holomorphic line bundle $L$ on $\Omega$, we fix our notational conventions as follows. 
\begin{itemize}
\item For a function $f\colon \Omega\to \mathbb{C}$, ${\rm Re}\,f$ (resp. ${\rm Im}\,f$) denotes the real part (resp. the imaginary part) of $f$ ($\mathbb{C}$ denotes the field of complex numbers). 
\item $T_\Omega:=T^{1, 0} \Omega$ (resp. $T_\Omega^*:=T^*\Omega$) denotes the holomorphic tangent (resp. cotangent) bundle of $\Omega$. 
\item $K_\Omega:=\textstyle\bigwedge^nT_\Omega^*$ denotes the canonical line bundle of $\Omega$. 
\item For a submanifold $Y\subset \Omega$, $N_{Y/\Omega}$ denotes the (holomorphic) normal bundle of $Y$ in $\Omega$. 
\item For a divisor $D$ of $\Omega$, $[D]$ denotes the holomorphic line bundle on $\Omega$ associated to $D$ (See \cite[Chapter I \S 1]{GH} for example). 
\item $L^{-1}$ denotes the dual line bundle of $L$. 
\item $\mathcal{O}_\Omega$ (resp. $\mathcal{O}_\Omega^*$) denotes the sheaf of germs of holomorphic functions (resp. nowhere vanishing holomorphic functions) on $\Omega$. 
\item $\mathcal{O}_\Omega(L)$ denotes the sheaf of germs of holomorphic sections of $L$. 
\item $\mathcal{A}_\Omega^{p, q}$ denotes the sheaf of germs of $C^\infty$ $(p, q)$-forms on $\Omega$. 
\item For an ideal sheaf $\mathcal{J}$ of $\mathcal{O}_\Omega$, $\mathcal{J}\mathcal{A}_\Omega^{p, q}$ denotes the image sheaf of the natural morphism $\mathcal{J} \otimes_{\mathcal{O}_\Omega}\mathcal{A}_\Omega^{p, q}\to \mathcal{A}_\Omega^{p, q}$. 
\item For a $C^\infty$ Hermitian metric $h$ on $L$, $\Theta_h$ denotes the Chern curvature form of $h$. 
\item $\delbar$ denotes the complex exterior derivative of type $(0, 1)$. 
\end{itemize}

\subsection{Some fundamental facts on cohomologies}\label{subsection:prelimcoh}

Let $Z$ be a complex manifold or a (reduced) analytic subset of a complex manifold, and $\mathcal{F}$ be a sheaf of abelian groups on $Z$. 
We denote by 
$H^q(Z, \mathcal{F})$ the $q$-th cohomology group with values in $\mathcal{F}$, which is the $q$-th derived functor of the functor $\Gamma(Z, -)$ of global sections evaluated on $\mathcal{F}$. 
The $q$-th cohomology group with values in $\mathcal{F}$ with compact support, which is defined as the $q$-th derived functor of the the functor $\Gamma_c(Z, -)$ of global sections with compact support evaluated on $\mathcal{F}$, will be denoted by $H_c^q(Z, \mathcal{F})$. 
As $Z$ is paracompact, $H^q(Z, \mathcal{F})$ can be identified with the $q$-th \v{C}ech cohomology group
\[
\check{H}^q(Z, \mathcal{F}) = \varinjlim_\mathcal{U}\check{H}^q(\mathcal{U}, \mathcal{F}), 
\]
where $\mathcal{U}$ runs over all the open coverings of $Z$ and $\check{H}^q(\mathcal{U}, \mathcal{F}) = \check{Z}^q(\mathcal{U}, \mathcal{F})/\check{B}^q(\mathcal{U}, \mathcal{F})$ (Here $\check{Z}^q(\mathcal{U}, \mathcal{F})$ (resp. $\check{B}^q(\mathcal{U}, \mathcal{F})$ denotes the group of $q$-th \v{C}ech cocycles (resp. coboundaries), see \cite[Chapter III \S 2]{GPR} for example). 
When $Z$ is a complex manifold and $\mathcal{F}=\mathcal{O}_Z$, we often compute the cohomologies by using the soft resolution
\begin{equation}\label{resolution:soft_dolbeault}
0 \to \mathcal{O}_Z \to \mathcal{A}_Z^{0, 0} \to \mathcal{A}_Z^{0, 1} \to \mathcal{A}_Z^{0, 2} \to \cdots, 
\end{equation}
where the morphism $\mathcal{A}_Z^{0, q} \to \mathcal{A}_Z^{0, q+1}$ is the one defined by the complex exterior derivative $\delbar$ of type $(0, 1)$ (Recall that, as $Z$ is locally compact and $\sigma$-compact, any soft sheaf is both $\Gamma(Z, -)$-acyclic and $\Gamma_c(Z, -)$-acyclic. See \cite[\S IV Theorem 2.2]{I} and \cite[\S III Theorem 2.7]{I} for example). In other words, we often identify 
$H^q(Z, \mathcal{O}_Z)$ with the $q$-th Dolbeault cohomology group $H^{0, q}(Z)$ and 
$H_c^q(Z, \mathcal{O}_Z)$ with the group 
\[
H_c^{0, q}(Z) := \begin{cases}
{\rm Ker}(\delbar\colon \Gamma_c(Z, \mathcal{A}_Z^{0, 0})\to \Gamma_c(Z, \mathcal{A}_Z^{0, 1}))=\Gamma_c(Z, \mathcal{O}_Z) & (q=0) \\
\frac{{\rm Ker}(\delbar\colon \Gamma_c(Z, \mathcal{A}_Z^{0, q})\to \Gamma_c(Z, \mathcal{A}_Z^{0, q+1}))}{{\rm Image}(\delbar\colon \Gamma_c(Z, \mathcal{A}_Z^{0, q-1})\to \Gamma_c(Z, \mathcal{A}_Z^{0, q}))} & (q>0)
\end{cases}. 
\]
We topologize each $\Gamma_c(Z, \mathcal{A}_Z^{0, q})$ and its subspaces by the topology of the uniform convergence of all derivatives with uniformly bounded supports, and $H_c^{0, q}(Z)$ by the quotient topology. 

Let $\mathcal{J}\subset \mathcal{O}_Z$ be an ideal sheaf which is flat. Then, by applying the same argument to the soft resolution 
\[
0 \to \mathcal{J} \to \mathcal{J}\mathcal{A}_Z^{0, 0} \to \mathcal{J}\mathcal{A}_Z^{0, 1} \to \mathcal{J}\mathcal{A}_Z^{0, 2} \to \cdots  
\]
obtained by applying the (exact) functor $\mathcal{J}\otimes_{\mathcal{O}_Z}-$ to the sequence $(\ref{resolution:soft_dolbeault})$, one has
\begin{equation}\label{resolution:soft_dolbeault_general}
H^{q}(Z, \mathcal{J}) \cong \begin{cases}
\Gamma(Z, \mathcal{J}) & (q=0) \\
\frac{{\rm Ker}(\delbar\colon \Gamma(Z, \mathcal{J}\mathcal{A}_Z^{0, q})\to \Gamma(Z, \mathcal{J}\mathcal{A}_Z^{0, q+1}))}{{\rm Image}(\delbar\colon \Gamma(Z, \mathcal{J}\mathcal{A}_Z^{0, q-1})\to \Gamma(Z, \mathcal{J}\mathcal{A}_Z^{0, q}))} & (q>0)
\end{cases}. 
\end{equation}

Let $G$ be an abelian group such as the group $\mathbb{Z}$ of integers, $\mathbb{R}$ of real numbers, $\mathbb{C}$ of complex numbers, and ${\rm U}(1):=\{t\in\mathbb{C}\mid |t|=1\}$. 
We denote by $\underline{G}_Z$ the sheaf of germs of $G$-valued locally constant functions on $Z$, where $Z$ is a complex manifold or an analytic subset of a complex manifold. 
As it is known that $Z$ is triangulable \cite{L}, we can identify $H^q(Z, \underline{G}_Z) (\cong \check{H}^q(Z, \underline{G}_Z))$ with the singular cohomology group $H^q(Z, G)$ \cite[Chapter IX Theorem 9.3]{ES}. 
For a complex manifold $\Omega$ and an analytic subset $Y\subset \Omega$, it again follows from \cite{L} that, for any neighborhood $V$ of $Y$ in $\Omega$, there exists an open neighborhood $V_0$ of $Y$ in $V$ such that 
\begin{equation}\label{isom:singularsheafcoh}
H^q(Y, \underline{G}_Y)
\cong H^q(Y, G)
\cong H^q(V_0, G)
\cong H^q(V_0, \underline{G}_{V_0})
\end{equation}
holds (Consider $V_0$ such that $Y$ is a deformation retract of $V_0$, see \cite[Proposition A.5]{H}). %$Z$ admits a structure of a CW-complex

Again let $\Omega$ be a complex manifold and $Y\subset \Omega$ be an analytic subset. Denote by $i\colon Y\to \Omega$ the inclusion. 
For a sheaf $\mathcal{F}$ on $V$, $\mathcal{F}|_Y$ denotes the (set-theoretical) restriction $i^{-1}\mathcal{F}$: i.e., $\mathcal{F}|_Y$ is the sheaf on $Y$ such that the stalk $(\mathcal{F}|_Y)_y$ coincides with $\mathcal{F}_y$ for any $y\in Y$. 
Note that the functor $\mathcal{F}\mapsto \mathcal{F}|_Y$ between the categories of sheaves on $\Omega$ and $Y$ is exact, see \cite[\S II.4]{I} for example. 
In what follows, the group $\textstyle\varinjlim_{V}H^q(V, \mathcal{F})$ plays an important role, where $V$ runs over open neighborhoods of $Y$ in $\Omega$ unless otherwise described. 
For this group, the following is fundamental. 
\begin{lemma}\label{lem:prelim_limHq}
Let $\Omega$ be a complex manifold, $Y\subset \Omega$ be an analytic subset, and $\mathcal{F}$ be a sheaf on $\Omega$. Then the natural morphism 
\[
\varinjlim_{V}H^q(V, \mathcal{F}) \to H^q(Y, \mathcal{F}|_Y)
\]
is isomorphic for any $q$. 
When $\mathcal{F}=\underline{G}_{\Omega}$ for an abelian group $G$, the natural morphism 
\[
\varinjlim_{V}H^q(V, \underline{G}_V) \to H^q(Y, \underline{G}_Y)
\]
is isomorphic for any $q$. 
\end{lemma}

\begin{proof}
Refer to \cite[\S IV Proposition 2.4]{I} for the former half. The latter half follows by considering open neighborhoods $V_0$ such that $(\ref{isom:singularsheafcoh})$ holds. 
\end{proof}

For a non-compact complex manifold $M$, we often use the exact sequence
\begin{align}\label{exseq:long_ex_seq_H_cHlimH}
0 &\to H^{0, 0}(M) \to \varinjlim_K H^{0, 0}(M\setminus K)
\to H_c^{0, 1}(M) \to H^{0, 1}(M) \to \varinjlim_K H^{0, 1}(M\setminus K)\\
&\to H_c^{0, 2}(M) \to H^{0, 2}(M) \to \varinjlim_K H^{0, 2}(M\setminus K)\to \cdots, \nonumber
\end{align}
which is obtained by considering the short exact sequence of chain complexes
\[
0\to \Gamma_c(M, \mathcal{A}_M^{0, \bullet}) \to \Gamma(M, \mathcal{A}_M^{0, \bullet}) \to \varinjlim_K\Gamma(M\setminus K, \mathcal{A}_M^{0, \bullet}) \to 0, 
\]
where $K$ runs over all the compact subsets of $M$. 

\begin{remark}\label{rmk:explicit_repre_renketu}
Take $a \in \textstyle\varinjlim_K H^{0, q}(M\setminus K)$. 
By construction of the long exact sequence, the following holds for the image $b \in H_c^{0, q+1}(M)$ of $a$ by the connecting morphism in the sequence ($\ref{exseq:long_ex_seq_H_cHlimH}$): For a compact subset $K\subset M$ and a $\delbar$-closed $(0, q)$-form $\eta$ which represents $a$, $b$ coincides with the class $[\delbar \widetilde{\eta}]$ represented by $\delbar \widetilde{\eta}$, where $\widetilde{\eta}$ is a $(0, q)$-form on $M$ such that $\widetilde{\eta}=\eta$ holds on $M\setminus \widehat{K}$ for some compact subset $\widehat{K}\subset M$ with $K\subset \widehat{K}$. 
\end{remark}

\subsection{Unitary flatness of holomorphic line bundles and the $\ddbar$-Lemma}\label{subsection:kasiwarathm}

Let $\Omega$ be a complex manifold. 
In accordance with \cite{KT}, we say that ``{\it the $\ddbar$-lemma holds on $\Omega$}" if the following holds: 
for any given $d$-exact $C^\infty$ $(1, 1)$-form $\theta$ on $\Omega$, there exists a $C^\infty$ function $f$ on $\Omega$ such that $\theta=\ddbar f$ holds. 
It is well-known that the $\ddbar$-lemma holds on $\Omega$ when $\Omega$ is a compact K\"ahler manifold, see \cite[Chapter 1 \S 2]{GH} or \cite[Proposition 1.7.24]{Kob} for example. 

Let $Z$ be a complex manifold or an analytic subset of a complex manifold. 
We say that a holomorphic line bundle $L$ on $Z$ is {\it unitary flat} if there exist an open covering $\{U_j\}$ of $Z$ and a local trivialisation $e_j$ of $L$ on each $U_j$ such that $|e_j/e_k|\equiv 1$ holds on each $U_j\cap U_k$: i.e., $L$ is in the image of the morphism $I\colon \check{H}^1(Z, \underline{{\rm U}(1)}_Z)\to {\rm Pic}(Z)=\check{H}^1(Z, \mathcal{O}_Z^*)$ induced by the natural injection $\underline{{\rm U}(1)}_Z\to \mathcal{O}_Z^*$. 
When $Z$ is a complex manifold, it is well-known that this condition is equivalent to 
the existence of a {\it flat metric} $h$: i.e., an Hermitian metric $h$ such that $\Theta_h$ is identically zero (See \cite[Proposition 1.4.21]{Kob} for example). 
For any holomorphic line bundle $L$ on $Z$, it follows from the following lemma that the unitary flat structure on $L$ is unique (if exists) when $Z$ is compact. 

\begin{lemma}\label{lem:inj_unitaryflattohol} 
The morphism $I$ is injective if $Z$ is compact. 
\end{lemma}

\begin{proof}
Although this lemma can be shown in the same manner as in the case where $Z$ is a manifold (see \cite[Proposition 1 $(2)$]{U}), 
we here briefly describe the proof for the reader's convenience. 
By considering the commutative diagrams as in \cite[p. 585]{U}, the proof is reduced to showing the injectivity of the natural map $H^1(Z, \underline{\mathbb{R}}_Z)\to H^1(Z, \mathcal{O}_Z)$. 
Take an open covering $\{U_j\}$ of $Z$ and an element $\{(U_{jk}, r_{jk})\}\in\check{Z}^1(\{U_j\}, \underline{\mathbb{R}}_Z)$ such that $[\{(U_{jk}, r_{jk})\}]=0\in \check{H}^1(\{U_j\}, \mathcal{O}_Z)$ holds. 
By taking refinement if necessary, one can take an element $f_j\in \Gamma(U_j, \mathcal{O}_Z)$ for each $j$ such that $-f_j+f_k=r_{jk}$ holds on each $U_j\cap U_k$. As $\{(U_j, {\rm Im}\,f_j)\}$ defines a global pluriharmonic function on a compact space $Z$, it follows from the maximal principle that each ${\rm Im}\,f_j$ is locally constant. Therefore we infer from the open mapping theorem that each $f_j$ is also locally constant, from which the assertion follows because $\delta(\{(U_j, {\rm Re}\,f_j)\})=\{(U_{jk}, r_{jk})\}$. 
\end{proof}

When $Z$ is a compact K\"ahler manifold, it is known that the image of $I$ coincides with the kernel of the first Chern-class map $c_1^{\mathbb{R}}\colon H^1(Z, \mathcal{O}_Z^*)\to H^2(Z, \mathbb{R})$ (a theorem of Kashiwara, see \cite[\S 1]{U} for example). %, which implies that one can naturally identify $\mathcal{P}(X)$ with $H^1(X, {\rm U}(1))$. 
Note that the inclusion ${\rm Image}\,I\subset {\rm Ker}\,c_1^{\mathbb{R}}$ is clear by the fact that, for any holomorphic line bundle $L$ and a $C^\infty$ Hermitian metric $h$ on $L$, $\textstyle\frac{\sqrt{-1}}{2\pi}\Theta_h$ represents the (de Rham) cohomology class $c_1^{\mathbb{R}}(L)$. 
The other inclusion can be deduced form the $\ddbar$-lemma as follows: 
\begin{proposition}\label{prop:kasiwara}
Let $Z$ be a complex manifold. 
Assume that the $\ddbar$-lemma holds on $Z$. 
Then, for any holomorphic line bundle $L$ on $Z$, $L$ is unitary flat if $c_1^{\mathbb{R}}(L)=0$. 
\end{proposition}

\begin{proof}
Let $L$ be a holomorphic line bundle on $Z$ with $c_1^{\mathbb{R}}(L)=0$. 
Take a $C^\infty$ Hermitian metric $h_0$ on $L$. 
As $\theta := \textstyle\frac{\sqrt{-1}}{2\pi}\Theta_{h_0}$ represents the class $c_1^{\mathbb{R}}(L)$, it follows from the assumption that $\theta=\sqrt{-1}\ddbar f$ holds for a $C^\infty$ function $f$ on $Z$. 
As $\overline{\theta}=\theta$, we may assume that $f$ is $\mathbb{R}$-valued by replacing $f$ with $(f+\overline{f})/2$. 
As $h := h_0 \cdot \exp(2\pi f)$ is a flat metric on $L$, the assertion follows. 
\end{proof}

%%%%%%%%%%%%%%%%%%%%%%%%%%%%

%%%%%%%%%%%%%%%%%%%%%%%%%%%%
\section{Obstruction classes and some sufficient conditions for the FPLB}\label{section:formal_seiri}

Let $\Omega$ be a complex manifold, $\mathcal{I}\subset\mathcal{O}_\Omega$ be a coherent ideal sheaf which is non-trivial (i.e., $\mathcal{I}\not=0, \mathcal{O}_\Omega$), and $Y$ be the analytic subset defined by $Y={\rm Supp}\,\mathcal{O}_\Omega/\mathcal{I}$. 
As a slight generalization of $Y_\mu$'s as in \S 1, here we consider the complex space $Y_\mu^{\mathcal{I}}:=(Y, \mathcal{O}_\Omega/\mathcal{I}^{\mu+1})$ and regard it as 
{\it the $\mu$-th infinitesimal neighborhood of $Y$ with respect to $\mathcal{I}$}. 
The natural morphism $Y_\mu^{\mathcal{I}}\to \Omega$ will be denoted by $i_\mu^{\mathcal{I}}$. 
Note that $i_\mu^{\mathcal{I}}\colon Y_\mu^{\mathcal{I}}\to \Omega$ coincides with $i_\mu\colon Y_\mu\to \Omega$ if $\mathcal{I}$ is the defining ideal sheaf $\mathcal{I}_Y$ of $Y$. 
It is known that the formal complex space $Y_\infty^{\mathcal{I}}:=(Y,\  \textstyle\varprojlim_\mu\mathcal{O}_{\Omega}/\mathcal{I}^{\mu+1})$ coincides with 
the formal completion $Y_\infty$ of $\Omega$ along $Y$ \cite[Lemme 2.1]{Ba}. %, which is a simple consequence of (analytic) Hilbert's Nullstellensatz when $Y$ is compact. %, since $\mathcal{I}_Y^N\subset \mathcal{I}\subset \mathcal{I}_Y$ holds for some positive $N$ in this case. 
% in what follows we always assume that there exists a positive integer $N$ such that 
%\begin{equation}\label{incl:comparison_I_IY}
%$\mathcal{I}_Y^N\subset \mathcal{I}\subset \mathcal{I}_Y$ holds in order %to avoid problems related to ambiguity 
%to clarify the meaning of the condition ``$\mu>>1$'', which will be assumed in some assertions below. 
%Note that it simply follows from (analytic) Hilbert's Nullstellensatz that such $N$ exists if $Y$ is compact. 

In this section, for $\Omega, \mathcal{I}$, and $Y$, we make some formal (purely sheaf-theoretic) observation on the obstructions and give some fundamental sufficient conditions for the triviality of a given holomorphic line bundle on a neighborhood of $Y$, which are nothing but a slight generalizations of the results and observations in \cite[I \S 3, 4, 5]{Gri}, which is for the case where $Y$ is non-singular and $\mathcal{I}=\mathcal{I}_Y$. Especially, the arguments in \S \ref{subsection:coh_van_obst_FPLB} are based on \cite[I \S 5 (a)]{Gri}. See also \cite[Chapter II \S 4, Chapter VII]{GPR} and \cite[\S 1]{P}. 

\subsection{The FPLB Problem in terms of $\mathcal{I}$}

%Therefore we have that the FPLB holds on a neighborhood of $Y$ in $\Omega$ if and only if the triviality of $(i_\infty^{\mathcal{I}})^*\mathcal{O}_V(L)$ (as an invertible sheaf on $Y_\infty^{\mathcal{I}}$) implies the existence of a neighborhood $V_0$ of $Y$ such that $L|_{V_0}$ is holomorphically trivial for any holomorphic line bundle $L$ on a neighborhood $V$ of $Y$. 
First let us rephrase the the triviality of $i_\infty^*\mathcal{O}_V(L)$ in terms of  the map $(\ref{eq:naturalmapfromvarprojlimLrestr})$ as follows. 

\begin{lemma}\label{lem:rel_probs}
Let $\Omega, \mathcal{I}$, and $Y$ be as above. 
For any holomorphic line bundle $L$ on a neighborhood $V$ of $Y$ in $\Omega$, the following are equivalent: \\
$(i)$ $i_\infty^*\mathcal{O}_V(L)$ is trivial (as an invertible sheaf on $Y_\infty$). \\
$(ii)$ There exists an element $\widetilde{\sigma}\in \textstyle\varprojlim_\mu \Gamma(V, \mathcal{O}_V(L)\otimes \mathcal{O}_V/\mathcal{I}^\mu)$ whose image by the natural map 
\[
\varprojlim_\mu \Gamma(V, \mathcal{O}_V(L)\otimes \mathcal{O}_V/\mathcal{I}^\mu) \to \Gamma(V, \mathcal{O}_V(L)\otimes \mathcal{O}_V/\sqrt{\mathcal{I}})
\]
is a global nowhere vanishing holomorphic section of $L|_Y$. %$L|_{Y_0^{\mathcal{I}}}=(i_0^{\mathcal{I}})^*L$ on $Y$. 
\end{lemma}

\begin{proof}
It is shown in \cite[Lemme 2.3 $(ii)$]{Ba} that the natural morphism 
\[
(i_\infty^{\mathcal{I}})^*\mathcal{O}_V(L) \to \varprojlim_\mu\Gamma\left(\mathcal{O}_V(L)/\mathcal{I}^{\mu+1} \mathcal{O}_V(L)\right)
\]
is an isomorphism. 
As $\mathcal{O}_V(L)/\mathcal{I}^{\mu+1}\mathcal{O}_V(L)
\cong \mathcal{O}_V(L)\otimes \mathcal{O}_\Omega/\mathcal{I}^{\mu+1}$ holds for any $\mu\geq 0$, %by virtue of the flatness of $\mathcal{O}_V(L)$, 
it follows from \cite[Proposition 4.2]{GPR} that the natural map
\[
\Gamma(Y, (i_\infty^{\mathcal{I}})^*\mathcal{O}_V(L)) \to \varprojlim_\mu\Gamma(Y, \mathcal{O}_V(L)\otimes \mathcal{O}_\Omega/\mathcal{I}^{\mu+1})
\]
is isomorphic. 
Via this isomorphism, the map in the condition $(ii)$ can be naturally identified with the restriction map
\[
\Gamma(Y, (i_\infty^{\mathcal{I}})^*\mathcal{O}_V(L)) \ni \widetilde{\sigma}\mapsto \widetilde{\sigma}|_Y\in \Gamma(Y, \mathcal{O}_Y(L|_Y))
=\Gamma(V, \mathcal{O}_V(L)\otimes \mathcal{O}_V/\sqrt{\mathcal{I}}). 
\]
Therefore the condition $(ii)$ is equivalent to the condition that $(i_\infty^{\mathcal{I}})^*\mathcal{O}_V(L)$ is trivial, which is equivalent to the condition $(i)$ because $Y_\infty^{\mathcal{I}}$ coincides with $Y_\infty$ by virtue of \cite[Lemme 2.1]{Ba}. 
\end{proof}

\subsection{Obstruction classes and some fundamental sufficient conditions for the FPLB}\label{subsection:coh_van_obst_FPLB}

In this subsection, we denote by $(\mathcal{I}^\mu)^*$ the image sheaf of $\mathcal{I}^\mu$ by the morphism $\mathcal{O}_\Omega\to \mathcal{O}_\Omega^*$ defined by the exponential map $f\mapsto \exp(2\pi\sqrt{-1}f)$. Then the restriction of the exponential morphism
\[
\mathcal{I}^\mu|_Y \to (\mathcal{I}^\mu)^*|_Y
\]
is an isomorphism for any $\mu\geq 1$, since the inverse is given by considering the branch of the logarithm which maps $1$ to $0$. 
Therefore one has a commutative diagram
\[
\xymatrix{
   & 0 \ar[r]\ar[d] & \mathcal{I}^\mu|_Y\ar[r]_{\cong\ \ }\ar[d] &  (\mathcal{I}^\mu)^*|_Y\ar[r]\ar[d]& 0\\
  0 \ar[r] & \underline{\mathbb{Z}}_{\Omega}|_Y \ar[r] & \mathcal{O}_{\Omega}|_Y\ar[r]& \mathcal{O}_{\Omega}^*|_Y\ar[r]& 0
}
\]
of sheaves with exact rows. 
By considering the induced long exact sequences and applying Lemma \ref{lem:prelim_limHq}, we have a commutative diagram
\[
\xymatrix{
0\ar[r]\ar[d] & \varinjlim_V H^1(V, \mathcal{I}^\mu)\ar[rd]_{\alpha_{\mathcal{I}, \mu}}\ar[r]_{\cong\ \ }\ar[d] & \varinjlim_V H^1(V, (\mathcal{I}^\mu)^*)\ar[d]\ar[r]& 0\ar[d]\\
H^1(Y, \mathbb{Z})\ar[r] & \varinjlim_V H^1(V, \mathcal{O}_V)\ar[r]& \varinjlim_V H^1(V, \mathcal{O}_V^*)\ar[r]& H^2(Y, \mathbb{Z})
}
\]
with exact rows. Let us denote by $\alpha_{\mathcal{I}, \mu}$ the map $\textstyle\varinjlim_V H^1(V, \mathcal{I}^\mu)\to \textstyle\varinjlim_V H^1(V, \mathcal{O}_V^*)$ which appears in the diagram above. 

For a holomorphic line bundle $L$ on a neighborhood of $Y$, 
we denote by $[L]_{Y}$ the element of $\textstyle\varinjlim_V H^1(V, \mathcal{O}_V^*)$ which corresponds to (the germ of) $L$. 

\begin{lemma}\label{lem:iikae_Ltriv_fund}
For a holomorphic line bundle $L$ on a neighborhood $V$ of $Y$ and a positive integer $\mu$, the following hold: \\
$(i)$  $[L]_{Y}=1\in\textstyle\varinjlim_V H^1(V, \mathcal{O}_V^*)$ holds if and only if there exists a neighborhood $V_0$ of $Y$ in $V$ such that $L|_{V_0}$ is holomorphically trivial. \\
$(ii)$ $[L]_{Y} \in {\rm Image}\,\alpha_{\mathcal{I}, \mu}$ holds if and only if $(i_{\mu-1}^{\mathcal{I}})^*\mathcal{O}_V(L)$ is trivial (as an invertible sheaf on $Y_{\mu-1}^{\mathcal{I}}$). 
\end{lemma}

\begin{proof}
The assertion $(i)$ is clear from the definition. 
The assertion $(ii)$ follows from the commutativity of the diagram above, since $(i_{\mu-1}^{\mathcal{I}})^*\mathcal{O}_V(L)$ is trivial if and only if one can take a system of local trivializations of $L$ whose transitions are trivial modulo $\mathcal{I}^\mu$, which means that the transitions are sections of $(\mathcal{I}^\mu)^*$ because $1+\mathcal{I}^\mu=(\mathcal{I}^\mu)^*$ holds as subsheaf of $\mathcal{O}_\Omega^*$. %(Here we applied the fundamental fact that the function $z\mapsto -1 + \exp(2\pi\sqrt{-1}z)$ maps a neighborhood of $0$ in $\mathbb{C}$ to another neighborhood of $0$ in $\mathbb{C}$ biholomorphically). 
%, for any element $f$ of the maximal ideal $\mathfrak{m}_{\mathbb{C}, 0}$ of the local ring $\mathcal{O}_{\mathbb{C}, 0}\cong \mathbb{C}\{z\}$, there exists $g\in\mathfrak{m}_{\mathbb{C}, 0}^\mu$ such that $1+f(z)=\exp(2\pi\sqrt{-1}g(z))$ if and only if $f\in \mathfrak{m}_{\mathbb{C}, 0}^\mu$). 
\end{proof}

Let $L$ be a holomorphic line bundle on a neighborhood of $Y$ such that $(i_{\mu-1}^{\mathcal{I}})^*\mathcal{O}_V(L)$ is trivial ($\mu\geq 1$). 
In what follows let us consider the obstruction to the triviality of $(i_{\mu}^{\mathcal{I}})^*\mathcal{O}_V(L)$. 

By Lemma \ref{lem:iikae_Ltriv_fund} $(ii)$, one can take an element $\xi \in \alpha_{\mathcal{I}, \mu}^{-1}([L]_{Y})$. 
To investigate the extendability of $\xi$, consider the exact sequence 
\begin{equation}\label{exseq:def_umu}
\varinjlim_V H^1(V, \mathcal{I}^{\mu+1}) \to
\varinjlim_V H^1(V, \mathcal{I}^\mu) \to H^1(Y, \mathcal{I}^\mu/\mathcal{I}^{\mu+1}), 
\end{equation}
which is induced from the exact sequence $0 \to \mathcal{I}^{\mu+1}|_Y \to \mathcal{I}^\mu|_Y \to (\mathcal{I}^\mu/\mathcal{I}^{\mu+1})|_Y\to 0$ (Here we applied Lemma \ref{lem:prelim_limHq}). 
Denote by $u_\mu(\mathcal{I}, L; \xi)$ the image of $\xi$ by the map $\textstyle\varinjlim_V H^1(V, \mathcal{I}^\mu) \to H^1(Y, \mathcal{I}^\mu/\mathcal{I}^{\mu+1})$ which appears in $(\ref{exseq:def_umu})$, and set 
\[
U_\mu(\mathcal{I}, L) := \{ u_\mu(\mathcal{I}, L; \xi) \mid \xi \in \alpha_{\mathcal{I}, \mu}^{-1}([L]_{Y}) \} (\subset H^1(Y, \mathcal{I}^\mu/\mathcal{I}^{\mu+1})). 
\]
As we will see below, under some practical constraints it can be shown that $U_\mu(\mathcal{I}, L)$ is a singleton, in which case we simply denoted by $u_\mu(\mathcal{I}, L)$ the element of $U_\mu(\mathcal{I}, L)$ and say that {\it the $\mu$-th obstruction class is well-defined}. 

\begin{remark}\label{rmk:shortcomparison_Ueda_obstr}
When $Y$ is non-singular and $\mathcal{I}$ is the defining ideal sheaf of $Y$, it clearly follows from the construction that our obstruction classes coincide with those defined in \cite[\S 3.2]{K2021} up to multiplication by a non-zero constant. 
In particular, when $Y$ is a compact non-singular hypersurface of $\Omega$ with unitary flat normal bundle, 
$\mathcal{I}$ is the defining ideal sheaf of $Y$, 
and $L = [Y]\otimes \widetilde{N_{Y/X}}^{-1}$, our obstruction classes coincides with Ueda classes \cite{U} up to multiplication by a non-zero constant (Also refer to \cite{N}). Here $\widetilde{N_{Y/X}}$ is {\it the flat extension} of the normal bundle $N_{Y/X}$: i.e., $\widetilde{N_{Y/X}}$ is the unitary flat line bundle on a neighborhood of $Y$ in $\Omega$ whose restriction to $Y$ coincides with $N_{Y/X}$. The germ of such a unitary flat line bundle uniquely exists around $Y$ by virtue of Lemma \ref{lem:prelim_limHq} applied by letting $G={\rm U}(1)$. 
\end{remark}

\begin{lemma}\label{lem:well-def_obstretc}
Let $L$ be a holomorphic line bundle on a neighborhood of $Y$ such that $(i_{\mu-1}^{\mathcal{I}})^*\mathcal{O}_V(L)$ is trivial ($\mu\geq 1$). 
%For an element $\xi \in \alpha_\mu^{-1}([L]_{Y})$, 
%For a positive integer $\mu$ and a holomorphic line bundle $L$ on a neighborhood of $Y$, 
Then the following hold: \\
%$(i)$ $(i_{\mu-1}^{\mathcal{I}})^*\mathcal{O}_V(L)$ is trivial if and only if $U_\mu(\mathcal{I}, L)\not=\emptyset$. \\
$(i)$ $(i_{\mu}^{\mathcal{I}})^*\mathcal{O}_V(L)$ is trivial if and only if $0\in U_\mu(\mathcal{I}, L)$. \\
$(ii)$ When the $\mu$-th obstruction class is well-defined, $(i_{\mu}^{\mathcal{I}})^*\mathcal{O}_V(L)$ is trivial if and only if $u_\mu(\mathcal{I}, L)=0$. \\
$(iii)$ The $\mu$-th obstruction class is well-defined if $Y$ is compact and the natural map 
\[
\varinjlim_V H^0(V, \mathcal{O}_V) \to H^0(\Omega, \mathcal{O}_{\Omega}/\mathcal{I}^\mu)
\]
is surjective. 
\end{lemma}

\begin{proof}
The assertions $(i)$ and $(ii)$ are clear from Lemma \ref{lem:iikae_Ltriv_fund} $(ii)$ and the definitions and the exactness of the sequence $(\ref{exseq:def_umu})$. 
Assume that $Y$ is compact. 
Then it follows from the maximum principle that the exponential map $H^0(Y, \mathcal{O}_Y)\to H^0(Y, \mathcal{O}_Y^*)$ is surjective. 
Therefore, by considering the long exact sequence induced from the exponential short exact sequence on $Y$, we infer that the natural map $H^1(Y, \underline{\mathbb{Z}}_Y)\to H^1(Y, \mathcal{O}_Y)$ is injective. 
Thus the assertion $(iii)$ follows from Lemma \ref{lem:degoffreedom_xi} below, by considering 
the exact sequence 
\begin{equation}\label{exseq:degoffreedomxi}
\varinjlim_V H^0(V, \mathcal{O}_V) \to \varinjlim_V H^0(V, \mathcal{O}_V/\mathcal{I}^\mu)
\to \varinjlim_V H^1(V, \mathcal{I}^\mu) \to \varinjlim_V H^1(V, \mathcal{O}_V), 
\end{equation}
which is induced from the exact sequence $0 \to \mathcal{I}^{\mu}|_Y \to\mathcal{O}_\Omega|_Y \to (\mathcal{O}_\Omega/\mathcal{I}^\mu)|_Y\to 0$ (Here we again applied Lemma \ref{lem:prelim_limHq}). 
\end{proof}

\begin{lemma}\label{lem:degoffreedom_xi}
Let $\mu$ be a positive integer and $L$ be a holomorphic line bundle on a neighborhood of $Y$. 
Assume that the natural map $H^1(Y, \underline{\mathbb{Z}}_Y)\to H^1(Y, \mathcal{O}_Y)$ is injective. 
Then, for any elements $\xi$ and $\xi'$ of $\alpha_{\mathcal{I}, \mu}^{-1}([L]_{Y})$, $\xi-\xi'$ is an element of the kernel of the natural map 
$\textstyle\varinjlim_V H^1(V, \mathcal{I}^\mu) \to \textstyle\varinjlim_V H^1(V, \mathcal{O}_V)$. 
\end{lemma}

\begin{proof}
Denote by $\zeta$ the image of $\xi-\xi'$ by the natural map 
$\textstyle\varinjlim_V H^1(V, \mathcal{I}^\mu) \to \textstyle\varinjlim_V H^1(V, \mathcal{O}_V)$. 
As $\alpha_{\mathcal{I}, \mu}(\xi-\xi')=1$, it follows from the exponential exact sequence that $\zeta$ is in the image of the natural map $\textstyle\varinjlim_V H^1(V, \underline{\mathbb{Z}}_V)\to \textstyle\varinjlim_V H^1(V, \mathcal{O}_V)$. 
Therefore there exists an open neighborhood $V$ of $Y$, an open covering $\{V_j\}$ of $V$, $\{(V_{jk}, f_{jk})\}\in \check{Z}^1(\{V_j\}, \mathcal{I}^\mu)$ whose class coincides with $\xi-\xi'$, $\{(V_{jk}, n_{jk})\}\in \check{Z}^1(\{V_j\}, \underline{\mathbb{Z}}_V)$, and $h_j\in \Gamma(V_j, \mathcal{O}_\Omega)$ for each $j$ such that 
$n_{jk}=f_{jk}-h_j+h_k$ 
holds on each $V_j\cap V_k$. 
Note that the class $\lambda\in \textstyle\varinjlim_V H^1(V, \underline{\mathbb{Z}}_V)$ represented by $\{(V_{jk}, n_{jk})\}$ is mapped to $\zeta$ by the natural map.  
As $\mu\geq 1$ and $Y={\rm Supp}\,\mathcal{O}_\Omega/\mathcal{I}$, we have $n_{jk}=h_j+h_k$ on each $Y\cap V_j\cap V_k$. Therefore $\lambda|_Y\in H^1(Y, \underline{\mathbb{Z}}_Y)$ is in the kernel of the natural map $H^1(Y, \underline{\mathbb{Z}}_Y)\to H^1(Y, \mathcal{O}_Y)$, which is $0$ by the assumption. 
Thus $\lambda|_Y=0$. From Lemma \ref{lem:prelim_limHq}, we infer that $\lambda=0$. Therefore we have $\zeta=0$, since $\lambda\mapsto \zeta$. \end{proof}

\begin{remark}\label{rmk:van_of_N-impliesj^*triv}
Let $L$ be a holomorphic line bundle on a neighborhood of $Y$ such that $(i_{\mu-1}^{\mathcal{I}})^*\mathcal{O}_V(L)$ is trivial ($\mu\geq 1$). 
Then $U_\mu(\mathcal{I}, L) = \{0\}$ holds if $H^1(Y, \mathcal{I}^\mu/\mathcal{I}^{\mu+1})=0$, since $U_\mu(\mathcal{I}, L)$ is a non-empty subset of $H^1(Y, \mathcal{I}^\mu/\mathcal{I}^{\mu+1})$ by definition. 
Thus one can infer from Lemma \ref{lem:well-def_obstretc} $(i)$ that $(i_{\mu}^{\mathcal{I}})^*\mathcal{O}_V(L)$ is trivial in this case. 
By an inductive application of this argument, we immediately obtain the following fact: When $H^1(Y, \mathcal{I}^\mu/\mathcal{I}^{\mu+1})=0$ holds for any $\mu\geq 1$, $i_{\infty}^*\mathcal{O}_V(L)$ is trivial for any holomorphic line bundle $L$ on a neighborhood of $Y$ such that $(i_{0}^{\mathcal{I}})^*\mathcal{O}_V(L)$ is trivial (Recall $Y_\infty=Y_\infty^{\mathcal{I}}$ \cite[Lemme 2.1]{Ba}). 
It seems worthwhile to emphasize that the well-definedness of the obstruction classes plays an important role to apply this inductive argument. % to show the triviality of $i_{\infty}^*\mathcal{O}_V(L)$. 
See also the argument in the proof of Lemma \ref{lem:suff_cond_fplb} $(ii)$. 
\end{remark}

We give some fundamental sufficient conditions for a line bundle to be trivial on a neighborhood of $Y$ as follows. 
\begin{lemma}\label{lem:suff_cond_fplb}
For a holomorphic line bundle $L$ on a neighborhood of $Y$ and a positive integer $\mu_0$, there exists a neighborhood $V$ of $Y$ such that $L|_V$ is holomorphically trivial if either of the following conditions holds: \\
%the following hold: \\
$(i)$ $(i_{\mu_0-1}^{\mathcal{I}})^*\mathcal{O}_V(L)$ is trivial and $\textstyle\varinjlim_V H^1(V, \mathcal{I}^{\mu_0})=0$, or\\
$(ii)$ $(i_{\mu-1}^{\mathcal{I}})^*\mathcal{O}_V(L)$ is trivial and the $\mu$-th obstruction class is well-defined for any $\mu\geq \mu_0$, and the image of the natural map 
\[
\varprojlim_\mu \varinjlim_V H^1(V, \mathcal{I}^\mu) \to \varinjlim_V H^1(V, \mathcal{I}^{\mu_0})
\]
is $\{0\}$. 
\end{lemma}

\begin{proof}
By Lemma \ref{lem:iikae_Ltriv_fund} $(ii)$, one can take an element $\xi_{\mu_0} \in \alpha_{\mathcal{I}, \mu_0}^{-1}([L]_{Y})$ if $(i_{\mu_0-1}^{\mathcal{I}})^*\mathcal{O}_V(L)$ is trivial. 
As $\alpha_{\mathcal{I}, \mu_0}^{-1}([L]_{Y})\subset \textstyle\varinjlim_V H^1(V, \mathcal{I}^{\mu_0})$, $\xi_{\mu_0}$ is necessarily equal to zero under the condition $(i)$. Therefore we have $[L]_{Y}=\alpha_{\mathcal{I}, \mu_0}(0) = 1$. Thus the assertion under the condition $(i)$ follows from Lemma \ref{lem:iikae_Ltriv_fund} $(i)$. 

Assume the condition $(ii)$. 
Then it follows from Lemma \ref{lem:well-def_obstretc} $(ii)$ that $u_{\mu_0}(\mathcal{I}, L; \xi_{\mu_0})=u_{\mu_0}(\mathcal{I}, L)$ is equal to $0$. 
Therefore, by the exactness of the sequence $(\ref{exseq:def_umu})$, one can take an element $\xi_{\mu_0+1}\in \textstyle\varinjlim_V H^1(V, \mathcal{I}^{\mu_0+1})$ which is mapped to $\xi_{\mu_0}$ by the natural map. 
Inductively repeating this argument, one can construct an element 
$\xi_\infty \in \textstyle\varprojlim_\mu \varinjlim_V H^1(V, \mathcal{I}^\mu)$ which is mapped to $\xi_{\mu_0}$ by the natural map. 
Thus it follows from the assumption that $\xi_{\mu_0}=0$, from which the assertion follows by applying the same argument as above. 
\end{proof}
%%%%%%%%%%%%%%%%

%%%%%%%%%%%%%%%%
\section{Proof of a weak variant of Theorem \ref{thm:main} $(i)$ and definition of $v(\mathcal{I}, L)$}

\subsection{Proof of a weak variant of Theorem \ref{thm:main} $(i)$}\label{subsection:weakthmmainiandstr}

In order to make our strategy in the proof of Theorem \ref{thm:main} clear, let us first give a proof of (a part of) Theorem \ref{thm:main} $(i)$ under the additional assumption of the vanishing of $H^1(X, \mathcal{I}^\mu)$ for $\mu$ large: 
\begin{proposition}[{A weak variant of Theorem \ref{thm:main} $(i)$}]\label{prop:simple_thmmaini}
Let $X$ be a compact K\"ahler manifold, $\mathcal{I}$ be a non-trivial coherent ideal sheaf of $\mathcal{O}_X$, 
and $Y\subset X$ be the support of $\mathcal{O}_X/\mathcal{I}$. 
Assume that the natural map $H_c^{2}(M, \mathcal{O}_M)\to H^{2}(M, \mathcal{O}_M)$ is injective, where $M=X\setminus Y$. 
Assume also that $H^1(X, \mathcal{I}^\mu)=0$ holds for $\mu$ sufficiently large. 
Then the FPLB holds on neighborhoods $Y$ of $X$. 
\end{proposition}

\begin{proof}%[Proof of Proposition \ref{prop:simple_thmmaini}]
For a positive integer $\mu$, we have the following exact sequence by considering the Mayer--Vietoris sequence \cite[\S II 5.10]{I} and by applying the exactness of the functor $\textstyle\varinjlim$: 
\[
H^1(X, \mathcal{I}^\mu)
\to H^1(M, \mathcal{O}_M)\oplus \varinjlim_V H^1(V, \mathcal{I}^\mu)
\to \varinjlim_V H^1(V\setminus Y, \mathcal{O}_{V\setminus Y}). 
\]
From this exact sequence and the assumption, we infer that the natural map
\[
H^{0, 1}(M)\oplus \varinjlim_V H^1(V, \mathcal{I}^\mu)
\to \varinjlim_V H^{0, 1}(V\setminus Y)
\]
is injective when $\mu$ is sufficiently large. 
Therefore the natural map
\[
\varinjlim_V H^1(V, \mathcal{I}^\mu)\to {\rm Coker}(H^{0, 1}(M)\to \varinjlim_V H^{0, 1}(V\setminus Y))
\]
is injective for $\mu$ sufficiently large. 
By virtue of the exactness of the sequence $(\ref{exseq:long_ex_seq_H_cHlimH})$, 
it follows from the injectivity assumption for $H_c^{2}(M, \mathcal{O}_M)\to H^{2}(M, \mathcal{O}_M)$ that ${\rm Coker}(H^{0, 1}(M)\to \textstyle\varinjlim_V H^{0, 1}(V\setminus Y))=0$ (Here we note that
\[
\varinjlim_V H^{0, 1}(V\setminus Y) = \varinjlim_K H^{0, 1}(M\setminus K)
\]
holds, where $K$ runs over all the compact subset of $M$, since $M\setminus (V\setminus Y)=X\setminus V$ is a compact subset of $M$ for any open neighborhood $V$ of $Y$ and $M\setminus K=W\setminus Y$ holds for a neighborhood $W=X\setminus K$ for any compact subset $K$ of $M$). 
Thus $\textstyle\varinjlim_V H^1(V, \mathcal{I}^\mu)=0$ holds for $\mu$ sufficiently large, from which one can easily deduce the assertion by applying 
Lemma \ref{lem:suff_cond_fplb} $(i)$. 
\end{proof}

We will show Theorem \ref{thm:main} $(i)$ without assuming the vanishing of $H^1(X, \mathcal{I}^\mu)$ along the same strategy as the proof of Proposition \ref{prop:simple_thmmaini} above. 
For this purpose, we will interpret the triviality of a line bundle in terms of curvature forms and reduce the problem to the $\ddbar$-problem on a neighborhood of $Y$ by applying a similar argument as in the proof of Proposition \ref{prop:kasiwara} so that one can ``disregard'' the information of $H^1(X, \mathcal{I}^\mu)$ from the viewpoint of harmonic theory. 
Following this strategy, it seems natural to regard (a suitable element of) the kernel of $H_c^{2}(M, \mathcal{O}_M)\to H^{2}(M, \mathcal{O}_M)$ as an obstruction for the triviality of a line bundle on a neighborhood of $Y$. 
Based on this observation, let us define an obstruction class $v(\mathcal{I}, L) \in {\rm Ker}(H_c^{2}(M, \mathcal{O}_M)\to H^{2}(M, \mathcal{O}_M))$ for a line bundle $L$ on a neighborhood of $Y$. 

\subsection{Definition of $v(\mathcal{I}, L)$}\label{section:def_obstr_class_vL}

Let $X$ be a compact complex manifold, $\mathcal{I}$ be a non-trivial coherent ideal sheaf of $\mathcal{O}_X$, and $Y\subset X$ be the support of $\mathcal{O}_X/\mathcal{I}$. 
Let $L$ be a holomorphic line bundle on a neighborhood of $Y$ such that $(i_{0}^{\mathcal{I}})^*\mathcal{O}_V(L)$ is trivial. 
For an element $\xi \in \alpha_{\mathcal{I}, 1}^{-1}([L]_{Y}) (\subset \textstyle\varinjlim_V H^1(V, \mathcal{I}))$, whose existence is ensured by Lemma \ref{lem:iikae_Ltriv_fund} $(ii)$, denote by $v(\mathcal{I}, L) \in H_c^{0, 2}(M)$ the image of $\xi$ by the composition 
\[
\varinjlim_V H^1(V, \mathcal{I})
\to \varinjlim_V H^1(V\setminus Y, \mathcal{O}_{V\setminus Y})
\to H_c^{0, 2}(M)
\]
of the restriction $\textstyle\varinjlim_V H^1(V, \mathcal{I})
\to \textstyle\varinjlim_V H^1(V\setminus Y, \mathcal{O}_{V\setminus Y})$ and the map 
$\textstyle\varinjlim_V H^{0, 1}(V\setminus Y)\to H_c^{0, 2}(M)$ which appears in the sequence $(\ref{exseq:long_ex_seq_H_cHlimH})$ (Recall the note in the proof of Proposition \ref{prop:simple_thmmaini}). 

\begin{lemma}\label{lem:obst_v_welldef}
The class 
$v(\mathcal{I}, L)$ is an element of ${\rm Ker}(H_c^{0, 2}(M)\to H^{0, 2}(M))$, and %depends only on $X, \mathcal{I}$, and $L$, and 
does not depend on the choice of $\xi \in \alpha_{\mathcal{I}, 1}^{-1}([L]_{Y})$. 
\end{lemma}

\begin{proof}
The former half of the assertion follows from 
the exactness of the sequence $(\ref{exseq:long_ex_seq_H_cHlimH})$. 
The latter half follows from Lemma \ref{lem:degoffreedom_xi}, since the restriction map $\textstyle\varinjlim_V H^1(V, \mathcal{I})
\to \textstyle\varinjlim_V H^1(V\setminus Y, \mathcal{O}_{V\setminus Y})$ factors through the natural map $\textstyle\varinjlim_V H^1(V, \mathcal{I}) \to \textstyle\varinjlim_V H^1(V, \mathcal{O}_V)$. 
\end{proof}

%%%%%%%%%%%%%%%%%%%%%%%%%%%%

%%%%%%%%%%%%%%%%%%%%%%%%%%%%
\section{Proof of Theorem \ref{thm:main} and its corollaries}\label{section:main_prf}

Here we show Theorem \ref{thm:main} and its corollaries (Corollaries \ref{cor:main1} and \ref{cor:main2}). 
The strategy of the proof is based on the observation we made in \S \ref{subsection:weakthmmainiandstr}. 
As a preparation, we first construct a suitable Hermitian metric on a holomorphic line bundle on a neighborhood of an analytic subset whose restriction to the $\mu$-th infinitesimal neighborhood is trivial in \S \ref{subsection:main_prf_pre1}. By using it, we investigate closer the obstruction class $v(\mathcal{I}, L)$ in \S \ref{subsection:main_prf_pre2}. 
By applying them, the proof of the theorem and corollaries will be given in \S \ref{subsection:main_prf_main}. 

\subsection{Hermitian metrics on a line bundle whose restriction to the $\mu$-th infinitesimal neighborhood is trivial}\label{subsection:main_prf_pre1}

In this subsection we show the following: 
\begin{lemma}\label{lem:refmetric}
Let $\mu$ be a positive integer, 
$\Omega$ a complex manifold, $Y$ a non-empty compact analytic subset of $X$, $\mathcal{I}\subset \mathcal{O}_\Omega$ a coherent ideal sheaf such that the support of $\mathcal{O}_\Omega/\mathcal{I}$ is $Y$, 
and $L$ be a holomorphic line bundle on a neighborhood $V$ of $Y$ in $X$. 
Assume that $(i_{\mu-1}^{\mathcal{I}})^*\mathcal{O}_V(L)$ is trivial. 
Then, for any element $\xi \in \alpha_{\mathcal{I}, \mu}^{-1}([L]_Y)$, 
there exist a neighborhood $W$ of $Y$ in $V$, a $C^\infty$ Hermitian metric $h_\mu$ on $L|_{W}$, and a $\delbar$-closed form $\eta_\mu \in \Gamma(W, \mathcal{I}^{\mu}\mathcal{A}_\Omega^{0, 1})$ such that 
$\sqrt{-1}\Theta_{h_\mu} = d(\eta_\mu + \overline{\eta_\mu})$ holds, and that the image of $-2\pi\xi$ by the natural map $\textstyle\varinjlim_{U}H^1(U, \mathcal{I}^\mu)\to \textstyle\varinjlim_{U}H^1(U, \mathcal{O}_U)=\textstyle\varinjlim_{U}H^{0, 1}(U)$ coincides with the restriction of the Dolbeault cohomology class $[\eta_\mu]\in H^{0, 1}(W)$. 
\end{lemma}

\begin{proof}
Take a sufficiently small open neighborhood $V^*$ of $Y$ in $V$, a sufficiently fine finite open covering $\{V_j\}_{j=1}^N$ of $V$, $\{(V_{jk}, \xi_{jk})\}\in \check{Z}^1(\{V_j\}, \mathcal{I}^\mu)$ such that 
$\xi=[\{(V_{jk}, \xi_{jk})\}]$, and a partition of unity $\{(V_j, \rho_j)\}_{j=1}^N\cup\{(\Omega\setminus \overline{V^*}, \rho_\infty)\}$ subordinate to the covering $\{V_j\}_{j=1}^N\cup\{\Omega\setminus \overline{V^*}\}$ of $\Omega$. 
As $\alpha_{\mathcal{I}, \mu}(\xi)=[L]_Y$, there exists a system $\{(V_j, s_j)\}$ of local trivializations of $L$ such that $s_k/s_j=\exp(2\pi\sqrt{-1}\xi_{jk})$ holds on each $V_j\cap V_k$ (by shrinking $V$ if necessary). 
Set 
$\sigma := \textstyle\sum_{j=1}^N \rho_j\cdot s_j$. 
Then one has
\[
\sigma = s_j\cdot \sum_{k=1}^N \rho_k\cdot \frac{s_k}{s_j}
=s_j\cdot \sum_{k=1}^N \exp(2\pi\sqrt{-1}\xi_{jk})\cdot \rho_k
=s_j\cdot \sum_{k=1}^N (1+F_{jk})\cdot \rho_k
=s_j\cdot \left(1+\sum_{k=1}^N F_{jk}\cdot \rho_k\right)
\]
on each $V_j$, where 
\[
F_{jk} := \exp(2\pi\sqrt{-1}\xi_{jk}) - 1 \in \Gamma(V_j\cap V_k,\ \mathcal{I}^\mu). 
\]
Therefore one can take a neighborhood $W$ of $Y$ in $V^*$ in which $\sigma$ is a nowhere vanishing $C^\infty$ section of $L|_W$. 
Let $h_\mu$ be the Hermitian metric on $L|_W$ defined by letting 
$|\sigma|_{h_\mu} = 1$ at any point of $W$. 
From 
\[
-\log |s_j|_{h_\mu}^2
 = \log \frac{1}{|s_j|_{h_\mu}^2} = \log \frac{|\sigma|_{h_\mu}^2}{|s_j|_{h_\mu}^2}
= \log \left|\frac{\sigma}{s_j}\right|^2
=\log \left|1 + \sum_k F_{jk}\cdot \rho_k\right|^2, 
\]
we infer that 
\[
\sqrt{-1}\Theta_{h_\mu}|_{W\cap V_j} = \sqrt{-1} \ddbar (\beta_j + \overline{\beta_j})
\]
holds, where 
\[
\beta_j := \log \frac{\sigma}{s_j} = \log \left(1 + \sum_k F_{jk}\cdot \rho_k\right) \in \Gamma(W\cap V_j,\ \mathcal{I}^\mu\mathcal{A}_\Omega^{0, 0}). 
\]
Here we are using the branch of the logarithm which satisfies $\log 1 = 0$ by shrinking $W$ if necessary. 
As
\[
\beta_j - \beta_k = \log \frac{\sigma}{s_j} - \log \frac{\sigma}{s_k} = \log \frac{s_k}{s_j} = 2\pi\sqrt{-1}\xi_{jk}, 
\]
we have that $\{(W\cap V_j, \delbar \beta_j)\}$ defines a $\delbar$-closed global $(0, 1)$-form $\gamma\in \Gamma(W, \mathcal{I}^\mu\mathcal{A}_\Omega^{0, 1})$ whose Dolbeault cohomology class is the counterpart of $2\pi\sqrt{-1}\xi$ via the \v{C}ech--Dolbeault correspondence. Therefore the assertion follows by letting $\eta_\mu = \sqrt{-1}\gamma$, since 
\[
\sqrt{-1} \ddbar (\beta_j + \overline{\beta_j})
=\sqrt{-1}d(\delbar\beta_j - \del \overline{\beta_j})
=d\left(\sqrt{-1}\,\delbar\beta_j + \overline{\sqrt{-1}\,\delbar\beta_j} \right)
\]
holds on each $W\cap V_j$. 
\end{proof}

\subsection{The $\ddbar$-problem on a neighborhood $Y$ and the obstruction class $v(\mathcal{I}, L)$}\label{subsection:main_prf_pre2}
From this subsection we will work on a compact K\"ahler manifold, since the following proposition plays a key role in our proof below. 
\begin{proposition}\label{prop:key}
Let $X$ be a compact K\"ahler manifold, $Y$ an analytic subset of $X$, $W$ an open neighborhood of $Y$ in $X$, and $\eta$ a $\delbar$-closed $(0, 1)$-form on $W$. For $\theta := d(\eta + \overline{\eta})$, there exist an open neighborhood $V_0$ of $Y$ in $W$ and a $C^\infty$ function $f\colon V_0\to \mathbb{R}$ such that $\theta=\sqrt{-1}\ddbar f$ holds on $V_0$ if the Dolbeault cohomology class $[\eta]\in H^{0, 1}(W)$ is mapped to $0\in H_c^{0, 2}(X\setminus Y)$ by the composition 
\[
H^{0, 1}(W)
\to \varinjlim_UH^{0, 1}(U)
\to \varinjlim_UH^{0, 1}(U\setminus Y)
\to H_c^{0, 2}(X\setminus Y)
\] 
of the natural maps: i.e., the composition of the natural map 
$H^{0, 1}(W)\to \textstyle\varinjlim_UH^{0, 1}(U)$, 
the restriction map 
$\textstyle\varinjlim_UH^{0, 1}(U)
\to \textstyle\varinjlim_UH^{0, 1}(U\setminus Y)$, 
and the map 
$\textstyle\varinjlim_UH^{0, 1}(U\setminus Y)
\to H_c^{0, 2}(X\setminus Y)$ which appears in the sequence $(\ref{exseq:long_ex_seq_H_cHlimH})$ with $M=X\setminus Y$. 
\end{proposition}

\begin{proof}
Take a $C^\infty$ function $\chi\colon X\to \mathbb{R}$ such that $0\leq \chi\leq 1$ at any point of $X$, $\chi\equiv 1$ holds on a (small) neighborhood $W_0$ of $Y$, and that the support of $\chi$ is included in $W$. 
From the assumption and Remark \ref{rmk:explicit_repre_renketu}, it follows that the class 
$[\delbar(\chi\cdot \eta)] \in H_c^{0, 2}(X\setminus Y)$ is trivial, and hence there exists a $C^\infty$ $(0, 1)$-form $\Phi$ on $X\setminus Y$ with compact support such that $\delbar\Phi=\delbar(\chi\cdot \eta)$. 
Set 
\[
\widetilde{\eta}:=\begin{cases}
\eta & \text{on}\ V_0 := W_0\setminus {\rm Supp}\,\Phi\\
\chi\cdot \eta-\Phi & \text{on}\ X\setminus Y
\end{cases}, 
\]
where ${\rm Supp}\,\Phi$ denotes the support of $\Phi$. 
Note that $\widetilde{\eta}$ is a $C^\infty$ $\delbar$-closed $(0, 1)$-form on $X$ by construction. 
Fixing a K\"ahler metric on $X$, let $\widetilde{\eta} = h + \delbar g$ be the decomposition in accordance with the decomposition 
\[
{\rm Ker}(\delbar\colon \Gamma(X, \mathcal{A}_X^{0, 1}) \to \Gamma(X, \mathcal{A}_X^{0, 2})) = \mathbb{H}_X^{0, 1} \oplus {\rm Image}(\delbar\colon \Gamma(X, \mathcal{A}_X^{0, 0}) \to \Gamma(X, \mathcal{A}_X^{0, 1}))
\]
of the space of $\delbar$-closed $(0, 1)$-forms on $X$ into the direct sum of the space $\mathbb{H}_X^{0, 1}$ of harmonic $(0, 1)$-forms on $X$ and the space of $\delbar$-exact $(0, 1)$-forms on $X$ (See \cite[Chapter 0 \S 6]{GH} for example), namely $h\in \mathbb{H}_X^{0, 1}$ and $g\in \Gamma(X, \mathcal{A}_X^{0, 0})$. 
By the Hodge identity between Laplacians (See \cite[Chapter 0 \S 7]{GH} for example), it follows that $h$ is $d$-closed: $dh=0$. Therefore we have 
\[
d\widetilde{\eta} = d(h + \delbar g) = dh + \ddbar g = \ddbar g. 
\]
Thus the assertion follows by letting $f := 2{\rm Re}(-\sqrt{-1}g)$ on $V_0$. 
\end{proof}

In the rest of this subsection, we let $X$ be a compact K\"ahler manifold, $Y$ a non-empty analytic subset of $X$, $\mathcal{I}\subset \mathcal{O}_\Omega$ a coherent ideal sheaf such that the support of $\mathcal{O}_\Omega/\mathcal{I}$ is $Y$, 
and $L$ be a holomorphic line bundle on a neighborhood $W$ of $Y$ in $X$ such that $(i_{0}^{\mathcal{I}})^*\mathcal{O}_W(L)$ is trivial. 
By applying Proposition \ref{prop:key} to $\eta_\mu$ as in 
Lemma \ref{lem:refmetric} and the argument as in the proof of Proposition \ref{prop:kasiwara}, we have the following: 
\begin{proposition}\label{prop:van_v_implies_Ltriv}
Let $X, Y, \mathcal{I}, W$, and $L$ be as above. 
Then 
$v(\mathcal{I}, L)=0$ holds if and only if $L|_{V_0}$ is holomorphically trivial for a neighborhood $V_0$ of $Y$ in $W$. 
\end{proposition}

\begin{proof}
When $L|_{V_0}$ is holomorphically trivial for a neighborhood $V_0$ of $Y$, $0 \in \alpha_{\mathcal{I}, 1}^{-1}([L]_Y)$ holds because $[L]_Y=1$ by virtue of Lemma \ref{lem:iikae_Ltriv_fund} $(i)$. 
Thus we have $v(\mathcal{I}, L)=0$ in this case. 

Assume $v(\mathcal{I}, L)=0$. 
Fix an element $\xi \in \alpha_{\mathcal{I}, 1}^{-1}([L]_Y)$ (Recall Lemma \ref{lem:iikae_Ltriv_fund} $(ii)$). 
By shrinking $W$, 
take a $C^\infty$ Hermitian metric $h_1$ on $L$ and a $\delbar$-closed form $\eta_1 \in \Gamma(W, \mathcal{I}\mathcal{A}_X^{0, 1})$ as in 
Lemma \ref{lem:refmetric}. 
As the image of the class of $-2\pi\xi$ by the natural map $\textstyle\varinjlim_{U}H^1(U, \mathcal{I})\to \textstyle\varinjlim_{U}H^1(U, \mathcal{O}_U)=\textstyle\varinjlim_{U}H^{0, 1}(U)$ coincides with the restriction of the Dolbeault cohomology class $[\eta_1]\in H^{0, 1}(W)$, 
it follows from the definition of $v(\mathcal{I}, L)$ and the assumption that the restriction of $[\eta_1]\in H^{0, 1}(W)$ is mapped to $0$ by the natural map $\textstyle\varinjlim_UH^{0, 1}(U\setminus Y)\to H_c^{0, 2}(X\setminus Y)$. 
Therefore we infer from Proposition \ref{prop:key} that there exist a neighborhood $V_1$ of $Y$ in $W$ and a $C^\infty$ real-valued function $f$ on $V_1$ such that $\sqrt{-1}\Theta_{h_1}=\sqrt{-1}\ddbar f$ holds. 
Let $h_{V_1}$ be the Hermitian metric on $L|_{V_1}$ defined by 
$h_{V_1} := h_1 \cdot \exp f$. 
As 
\[
\sqrt{-1}\Theta_{h_{V_1}} = \theta - \sqrt{-1}\ddbar f \equiv 0, 
\]
$h_{V_1}$ is a flat metric and thus $L|_{V_1}$ admits a structure as a unitary flat line bundle. 
Take a neighborhood $V_0$ of $Y$ in $V_1$ such that $(\ref{isom:singularsheafcoh})$ holds for $G={\rm U}(1)$. 
Then it follows that $L|_{V_0}$ (as a unitary flat line bundle) is trivial because $L|_Y$ is trivial, from which the assertion follows. 
\end{proof}

\begin{remark}\label{rmk:flatholinj}
In the last part of the proof of Proposition \ref{prop:van_v_implies_Ltriv}, we applied %the fact that the natural map $H^1(Y, \underline{{\rm U}(1)}_Y)\to H^1(Y, \mathcal{O}_Y^*)$ is injective 
Lemma \ref{lem:inj_unitaryflattohol} 
in order to ensure the triviality of $L|_Y$ as a unitary flat line bundle. 
Note that the triviality of $L|_Y$ as a holomorphic line bundle is clear by considering the image of the global trivialization of $(i_0^{\mathcal{I}})^*\mathcal{O}_V(L)$ by the natural map $\Gamma(Y, \mathcal{O}_V(L)\otimes \mathcal{O}_V/\mathcal{I})\to \Gamma(Y, \mathcal{O}_V(L)\otimes \mathcal{O}_V/\sqrt{\mathcal{I}})$. 
\end{remark}

We also have the following as an application of Lemma \ref{lem:refmetric}: 
\begin{lemma}\label{lem:van_v_jetnu}
Let $X, Y, \mathcal{I}$, and $L$ be as above, and $\mu$ be a positive integer. 
Assume that $\mathcal{I}=\mathcal{O}_X(-D)$ holds for an effective divisor $D$ on $X$, and that $(i_{\mu-1}^{\mathcal{I}})^*\mathcal{O}_V(L)$ is trivial. 
Then $v(\mathcal{I}, L)$ is mapped to $0$ by the natural map $H_c^{0, 2}(M)\to H^2(X, \mathcal{I}^\mu)$. 
\end{lemma}

\begin{proof}
Fix an element $\xi \in \alpha_{\mathcal{I}, \mu}^{-1}([L]_Y)$ (again recall Lemma \ref{lem:iikae_Ltriv_fund} $(ii)$). 
By shrinking $W$, take a $C^\infty$ Hermitian metric $h_\mu$ on $L$ and a $\delbar$-closed form $\eta_\mu \in \Gamma(W, \mathcal{I}^{\mu}\mathcal{A}_X^{0, 1})$ as in 
Lemma \ref{lem:refmetric}. 
As $\mathcal{I}^\mu=\mathcal{O}_X(-\mu D)$ is locally free, one can identify 
$H^2(X, \mathcal{I}^\mu)$ with $H^2(\Gamma(X, \mathcal{I}^\mu\mathcal{A}_X^{0, \bullet}))$ via $(\ref{resolution:soft_dolbeault_general})$. 
Via this identification, it is clear by definition that the image of $v(\mathcal{I}, L)$ by the map $H_c^{0, 2}(M)\to H^2(X, \mathcal{I}^\mu)$ is the class of $\delbar(\chi\cdot \eta_\mu)$, where $\chi$ is as in the proof of Proposition \ref{prop:key}. 
As $\chi\cdot \eta_\mu\in \Gamma(X, \mathcal{I}^\mu\mathcal{A}_X^{0, 1})$, we have the assertion. 
\end{proof}

\subsection{Proof of Theorem \ref{thm:main} and its corollaries}\label{subsection:main_prf_main}

\begin{proof}[Proof of Theorem \ref{thm:main}]
By Lemma \ref{lem:obst_v_welldef}, we infer $v(\mathcal{I}_Y, L)=0$ for any holomorphic line bundle $L$ on a neighborhood of $Y$ whose restriction $L|_Y$ is holomorphically trivial when $H_c^{0, 2}(M)\to H^{0, 2}(M)$ is injective. 
Therefore the assertion $(i)$ follows from Proposition \ref{prop:van_v_implies_Ltriv}. 

Assume that there exists an effective divisor $D$ on $X$ such that ${\rm Supp}\,D=Y$ and that the natural map $H_c^{2}(M, \mathcal{O}_M)\to \textstyle\varprojlim_\mu H^2(X, \mathcal{I}^{\mu})$ is injective, where $\mathcal{I}:=\mathcal{O}_X(-D)$. 
Take a holomorphic line bundle $L$ on a neighborhood $V$ of $Y$ such that $i_\infty^*\mathcal{O}_V(L)$ is trivial. 
Then it follows from Lemma \ref{lem:van_v_jetnu} that 
$v(\mathcal{I}, L)$ is mapped to $0$ by the natural map $H_c^{0, 2}(X\setminus Y)\to \textstyle\varprojlim_\mu H^2(X, \mathcal{I}^\mu)$. 
Thus we have $v(\mathcal{I}, L)=0$. Therefore the assertion $(ii)$ follows again from Proposition \ref{prop:van_v_implies_Ltriv}. 
\end{proof}

\begin{proof}[Proof of Corollary \ref{cor:main1}]
By the argument in \cite[p. 294]{D1}, one can deduce from a theorem of Ohsawa for very strongly $(n-2)$-convex K\"ahler manifolds \cite{O1982} that $H_c^{0, 2}(M) \to H^{0, 2}(M)$ is injective. 
Thus the corollary follows from Theorem \ref{thm:main} $(i)$. 
\end{proof}

\begin{proof}[Proof of Corollary \ref{cor:main2}]
In this proof, for each $q\geq 0$, we will topologize $H^{2, q}(M)$ by considering the resolution of $\mathcal{O}_M(K_M)$ by the sheaves of currents and giving the Fr\'echet topology of the uniform convergence on bounded sets (the strong dual topology) on the spaces of currents, see \cite[\S 1.4.3]{Obook} for example. 
Note that this topology on $H^{2, 0}(M)=\Gamma(M, \mathcal{O}_M(K_M))$ coincides with the topology of uniform convergence on compact subsets, by virtue of Banach's open mapping theorem. 

First let us show the assertion when the condition $(i)$ holds. 
By applying Serre's duality theorem \cite{Se} \cite{LL1} (refer also to \cite[Theorem 1.33]{Obook}), we infer from the Hausdorffness of $H_c^{0, 2}(M)$ that the natural map $H_c^{0, 2}(M) \to (H^{2, 0}(M))^\vee$ is injective, where $(H^{2, 0}(M))^\vee$ denotes the topological dual space of $H^{2, 0}(M)$. 
By considering the dual of the natural map $\textstyle\varinjlim_\mu \Gamma(X, \mathcal{O}_X(K_X)\otimes \mathcal{O}_X(\mu D))\to \Gamma(M, \mathcal{O}_M(K_M))$, it follows from the assumption that the natural map $(H^{2, 0}(M))^\vee\to \textstyle\varprojlim_\mu (H^0(X, \mathcal{O}_X(K_X)\otimes\mathcal{O}_X(\mu D))^\vee=\textstyle\varprojlim_\mu H^2(X, \mathcal{O}_X(-\mu D))$ is also injective. 
Therefore the assertion in this case follows from Theorem \ref{thm:main} $(ii)$.

Next let us show the assertion under the condition $(ii)$. %(See also Remark \ref{rmk:elemprf_cormain2ii} below). 
Take a $C^\infty$ Hermitian metric $h$ on $[D]$ whose Chern curvature form is positive definite at any point of $Y$. Then the function $\vp:=-\log |\sigma_D|_h^2$ is an exhaustion function on $M$ which is strictly plurisubharmonic outside a compact subset $K$ of $M$, where $\sigma_D$ is the canonical section of $[D]$. 
Therefore we infer that $D$ is pseudoconcave of order $2$ in the sense of \cite{O2021}, since $|\sigma_D|_h^{-2}=\exp\,\vp$ is also strictly plurisubharmonic on $M\setminus K$. 
By virtue of \cite[Theorem 0.3]{O2021}, the natural map $\textstyle\varinjlim_\mu \Gamma(X, \mathcal{O}_X(K_X)\otimes \mathcal{O}_X(\mu D))\to \Gamma(M, \mathcal{O}_M(K_M))$ has dense image. 
As $H^{2, 1}(M)$ is Hausdorff by Andreotti--Grauert theorem \cite{AG} (also refer to \cite[Lemme 2]{Se}, \cite[\S 2]{O1982}, and \cite[\S 4]{D1}), it follows from \cite[Theorem 1.1]{LL1} that $H_c^{0, 2}(M)$ is also Hausdorff, see also \cite[Theorem 1.33]{Obook}. 
Thus the assertion follows from Corollary \ref{cor:main2} $(i)$. 
\end{proof}

\begin{remark}\label{rmk:elemprf_cormain2ii}
Note that we may assume that the function $\vp$ in the proof of Corollary \ref{cor:main2} $(ii)$ is plurisubharmonic by applying the standard modification argument: for example by taking the regularized maximum function of $\vp$ and a sufficiently large constant, see \cite[\S 2.1]{K2019}. 
In particular, by considering the Remmert reduction, $M$ is a proper modification of a Stein space (See \cite[Chapter V \S 2]{GPR} for example). 
Let $p\colon X \to X'$ be the contraction of the maximal analytic subset of $M$. %Denote by $D'$ the effective divisor of $X'$ such that $p^*D'=D$. 
%As $p$ maps a neighborhood of $Y$ to a neighborhood of the support $Y'$ of $D'$ biholomorphically, it clearly holds that $[D']$ also admits an Hermitian metric whose Chern curvature form is positive definite at any point of $Y'$. 
%As $[D']|_{Y'}$ is positive and $X'\setminus Y'$ is Stein, one can apply
From \cite[p, 347 Satz 4]{G}, \cite[p. 343 Satz 2]{G}, and the assumption $(ii)$ of Corollary \ref{cor:main2}, it follows that $X'$ is a projective algebraic variety and that $p(M)$ can be naturally regarded as an affine domain of  $X'$ (Refer to \cite{KUed} for the details). 
Therefore, for example when $\Gamma(M, \mathcal{O}_M(K_M^{-1})\otimes \mathcal{O}_M(\mu D))\not=0$ for some integer $\mu$, one can also show the density of the image of $\textstyle\varinjlim_\mu \Gamma(X, \mathcal{O}_X(K_X)\otimes \mathcal{O}_X(\mu D))\to \Gamma(M, \mathcal{O}_M(K_M))$ only by applying fundamental arguments under the assumption $(ii)$ of Corollary \ref{cor:main2}. 
\end{remark}

\begin{remark}\label{rmk:hypersurfaceYstYY>0}
Let $X$ be a compact complex manifold and $Y$ be a non-singular curve in $X$ whose normal bundle $N_{Y/X}$ is positive. 
Then it follows from \cite[Proposition (2.2)]{Su} that the complement $M:=X\setminus Y$ is strongly $1$-convex. 
Therefore one can apply the same argument as in Remark \ref{rmk:elemprf_cormain2ii} to construct a proper map $p\colon X \to X'$ which maps a neighborhood of $Y$ to a neighborhood of $Y':=p(Y)$ biholomorphically such that $X'\setminus Y'$ is a Stein space. 
Again by applying \cite[p, 347 Satz 4]{G}, we have that $[Y']$ is a positive line bundle on $X'$. Therefore there exists a neighborhood $V$ of $Y$ in $X$ such that $[Y]|_V$ is positive (See also \cite[p. 440 Corollary]{Gri}). 
In summary, the condition $(ii)$ in Corollary \ref{cor:main2} is satisfied when $D=Y$ and $N_{Y/X}>0$. Refer to \cite{KUed} for the case where $Y$ is singular. 
\end{remark}

%%%%%%%%%%%%%%%%%%%%%%%%%%%%

%%%%%%%%%%%%%%%%%%%%%%%%%%%%
\section{Examples and Discussions}\label{section:exdiss}

In this section we compare our results with some known results. 
We also give some examples and prove Theorem \ref{thm:exp2bl9pts}. 

\subsection{Higher dimensional cases}

Let $\Omega$ be a complex manifold, $\mathcal{I}\subset \mathcal{O}_\Omega$ a non-trivial coherent ideal sheaf, and $Y$ be the support of $\mathcal{O}_\Omega/\mathcal{I}$. 
In what follows we always assume that there exists a positive integer $N$ such that $\mathcal{I}_Y^N\subset \mathcal{I}\subset \mathcal{I}_Y$ holds in order to clarify the meaning of the condition ``$\mu>>1$'', which stands for ``$\mu$ is sufficiently large'' and will be assumed in some assertions below. 
Note that the existence of such $N$ simply follows from (analytic) Hilbert's Nullstellensatz when $Y$ is compact. 

From Lemma \ref{lem:suff_cond_fplb} $(i)$ it follows that the FPLB holds on neighborhoods of $Y$ in $\Omega$ if 
\begin{equation}\label{eq:conclusion_Griff}
\mu >> 1 \implies \varinjlim_V H^1(V, \mathcal{I}^\mu) = 0
\end{equation}
holds. 
Moreover, in this case, {\it the stable uniqueness} (in the sense of \cite{Gri}) of a holomorphic line bundles holds on neighborhoods of $Y$ in $\Omega$: i.e., one can infer the triviality of a holomorphic line bundle on a neighborhood of $Y$ by showing the vanishing of finite numbers of obstructions. 
This observation was already made by Griffiths in \cite{Gri}. %(at least when $Y$ is a compact submanifold and $\mathcal{I}$ is the defining ideal sheaf of $Y$). 
Also Griffiths gave a sufficient condition for the condition $(\ref{eq:conclusion_Griff})$ \cite[Theorem III]{Gri} (See also \cite[VII \S 4]{GPR}). 
By assuming that the codimension of $Y$ is $1$, for simplicity, from this we have the following: 
\begin{theorem}[{A consequence of \cite[Theorem III]{Gri}, see also \cite[Proposition 2.6]{Gri}}]\label{thm:griffmain}
Let $\Omega$, $\mathcal{I}$, and $Y$ be as above. 
Assume that $Y$ is a compact non-singular hypersurface of $\Omega$ and there exist an integer $s>1$ and a $C^\infty$ Hermitian metric on the normal bundle $N_{Y/\Omega}$ of $Y$ whose Chern curvature form has $s$ positive and ${\rm dim}\,Y-s$ negative eigenvalues at any point of $Y$. 
Then %there exists a positive integer $\mu_0$ such that, for any holomorphic line bundle $L$ on a neighborhood $V$ of $Y$ in $X$, there exists a neighborhood $V_0$ of $Y$ in $X$ such that $L|_{V_0}$ is holomorphically trivial if $(i_{\mu_0}^\mathcal{I})^*\mathcal{O}_V(L)$ is trivial. 
%In particular, 
the FPLB holds on neighborhoods of $Y$ in $\Omega$. 
\end{theorem}

\begin{proof}
The assertion follows from \cite[Theorem III]{Gri} %Hilbert's Nullstellensatz, 
and Lemma \ref{lem:suff_cond_fplb} $(i)$ 
by applying the argument above. 
\end{proof}

From Theorem \ref{thm:griffmain}, we have that the FPLB holds on neighborhoods of $Y$ in $\Omega$ especially when ${\rm dim}\,\Omega\geq 3$ and $Y$ is a compact non-singular hypersurface of $\Omega$ whose normal bundle is positive. 
Corollary \ref{cor:main2} $(ii)$ can be regarded as a generalization of this fact to the case where $\Omega$ can be realized as an open subset of a compact K\"ahler surface $X$ (Recall Remark \ref{rmk:hypersurfaceYstYY>0}). 
Note that the FPLB holds on neighborhoods of $Y$ in $\Omega$ when $Y$ is a compact non-singular hypersurface of $\Omega$ whose normal bundle is negative, since in this case $Y$ is an exceptional subset (in the sense of Grauert \cite{G}) and one can apply the following theorem due to Peternell: 
\begin{theorem}[{Part of a theorem of Peternell \cite{P}, see also \cite[Chapter VII. Theorem 4.5]{GPR}}]
Let $\Omega$, $\mathcal{I}$, and $Y$ be as above. 
Assume that $Y$ is an exceptional subspace. Then the FPLB holds on neighborhoods of $Y$ in $\Omega$. 
\end{theorem}
Note that, at least for the case where $Y$ is a compact non-singular hypersurface of $\Omega$ whose normal bundle is negative, for simplicity, 
the stable uniqueness of a holomorphic line bundles holds on neighborhoods of $Y$ in $\Omega$, since
\[
\mu >> 1 \implies H^1(Y, \mathcal{I}_Y^\mu/\mathcal{I}_Y^{\mu+1}) = H^1(Y, \mathcal{O}_Y(N_{Y/\Omega}^{-\mu})) = 0
\]
follows from Kodaira vanishing theorem. 
For some concrete examples, one can apply \cite[Theorem 4.5]{O1984} to estimate the lower bound of positive integers $\mu_0$ such that the triviality of $(i_{\mu_0}^{\mathcal{I}})^*\mathcal{O}_V(L)$ implies the triviality of $L$ on a neighborhood of $Y$ for any holomorphic line bundle $L$ on a neighborhood $V$ of $Y$, see \cite[Remark 3.2]{K2015} and \cite[Theorem 1.5, 4.5]{K2021} for example. 

For the case where the normal bundle is flat, the following is shown by Ohsawa on the complement of $Y$ in a compact K\"ahler manifold $X$: 
\begin{theorem}[{\cite[p. 115 Theorem]{O2007}}]\label{thm:O2007}
Let $X$ be a compact K\"ahler manifold of dimension $n\geq 3$ and $Y$ be a hypersurface of $X$. 
Then there does not exist an effective divisor $D$ of $X$ such that $[D]|_Y$ is topologically trivial if the complement $M:=X\setminus Y$ of $Y$ is very strongly $(n-2)$-convex. 
\end{theorem}

In what follows we see that one can deduce the following variant of this theorem from Corollary \ref{cor:main1}, in which the condition for $[D]|_Y$ is replaced with the unitary flatness. 
Here we remark that a theorem of Ohsawa \cite{O1982} plays key roles in the proofs of both Theorem \ref{thm:O2007} and Corollary \ref{cor:main1}.

\begin{proposition}\label{prop:variantofo2007}
Let $X$ be a compact K\"ahler manifold of dimension $n\geq 3$ and $Y$ be a hypersurface of $X$. 
Then there does not exist an effective divisor $D$ of $X$ such that $[D]|_Y$ is unitary flat if the complement $M:=X\setminus Y$ of $Y$ is very strongly $(n-2)$-convex. 
\end{proposition}

\begin{proof}
Let $X$ be a compact K\"ahler manifold and $Y$ be a hypersurface of $X$. 
Assuming that $M=X\setminus Y$ is very strongly $(n-2)$-convex, take a $C^\infty$ exhaustion plurisubharmonic function $\vp$ on $M$ whose Levi form has at least $3$ positive eigenvalues at any point of $W\setminus Y$ for a neighborhood $W$ of $Y$. 
Assume that there exists an effective divisor $D$ of $X$ such that $[D]|_Y$ is unitary flat. 
Denote by $F$ the flat extension of $[D]|_Y$ on a neighborhood $V$ of $Y$ in $X$: i.e., $F$ is the unitary flat line bundle on $V$ whose restriction to $Y$ coincides with $[D]|_Y$, see Remark \ref{rmk:shortcomparison_Ueda_obstr}. 
By applying Corollary \ref{cor:main1} by letting $L:=[D]\otimes F^{-1}$, one has that there exists a neighborhood $V_0$ of $Y$ in $V$ such that $[D]|_{V_0}$ is isomorphic to $F|_{V_0}$. 
In particular, $[D]|_{V_0}$ admits a structure as a flat line bundle. 
We may assume $V_0\subset W$ by shrinking $V_0$ if necessary. 

Take a flat metric $h$ on $[D]|_{V_0}$ and set $\psi:=-\log |\sigma_D|_h^2$, where $\sigma_D$ denotes the canonical section of $[D]$. 
Then, for a sufficiently large $m$, the levelset $H_m:=\{x\in V_0\mid \psi(x)=m\}$ is a compact Levi-flat hypersurface of $V_0$. 
As $H_m$ is compact and $\vp$ is continuous, there exists a point $p\in H_m$ at which the restriction $\vp|_{H_m}$ takes its maximum. 
Let $\mathcal{L}$ be the leaf of the Levi foliation of $H_m$ such that $p\in \mathcal{L}$. By applying 
the maximum principle for $\vp|_{\mathcal{L}}$, it follows that $\vp|_{\mathcal{L}}$ is constant. In particular, the Levi form of $\vp$ has at least $n-1$ zero eigenvalues at $p$, which is a contradiction since $p\in H_m\subset V_0\subset W$.
\end{proof}

\subsection{The FPLB on neighborhoods of a compact non-singular curve in a surface}\label{subsection:curvsurf}

Here we let $X$ be a connected non-singular complex surface and $Y$ be a connected non-singular hypersurface of $X$: i.e., $Y$ is a compact Riemann surface holomorphically embedded in $X$. 
As we have already seen that the FPLB holds on neighborhoods of $Y$ in $X$ when ${\rm deg}\,N_{Y/X}$ is negative (by a theorem of Peternell) and positive (by Corollary \ref{cor:main2} $(ii)$), we will investigate the case where ${\rm deg}\,N_{Y/X}=0$ in this subsection. 

Assume ${\rm deg}\,N_{Y/X}=0$. 
From a theorem of Kashiwara, it follows that $N_{Y/X}$ is unitary flat. 
Take a tubular neighborhood $V$ of $Y$ and denote by $U$ the Ueda line bundle $[Y]|_V\otimes \widetilde{N_{Y/X}}^{-1}$ (Refer to Remark \ref{rmk:shortcomparison_Ueda_obstr}). %, where $\widetilde{N_{Y/X}}$ is the flat extension of $N_{Y/X}$. 
In \cite{U}, Ueda classified the complex structure of a neighborhood of $Y$ in $X$ into the following four cases by focusing the (non-) triviality of the restriction of $U$ to the infinitesimal neighborhoods of $Y$, the formal completion of $X$ along $Y$, and to a small neighborhood of $Y$. 
\begin{description}
\item[($\alpha$)] The case where $(i_\mu^{\mathcal{I}_Y})^*\mathcal{O}_V(U)$ is not trivial for a positive integer $\mu$. 
\item[($\beta'$)] The case where $N_{Y/X}$ is a torsion element of ${\rm Pic}^0(Y)$ and $i_{\infty}^*\mathcal{O}_V(U)$ is trivial. 
\item[($\beta''$)] The case where $N_{Y/X}$ is a non-torsion element of ${\rm Pic}^0(Y)$ and $U|_{V_0}$ is holomorphically trivial on a neighborhood $V_0$ of $Y$. 
\item[($\gamma$)] The case where $N_{Y/X}$ is a non-torsion element of ${\rm Pic}^0(Y)$, $i_{\infty}^*\mathcal{O}_V(U)$ is trivial, and $U|_{V_0}$ is not holomorphically trivial for any neighborhood $V_0$ of $Y$. 
\end{description}
Note that the well-definedness of the $\mu$-th obstruction class is shown in \cite[\S 2.1]{U} when $(i_{\mu-1}^{\mathcal{I}_Y})^*\mathcal{O}_V(U)$ is trivial ($\mu\geq 1$). The class $u_\mu(\mathcal{I}_Y, U)\in H^1(Y, N_{Y/X}^{-\mu})$ is called the {\it $\mu$-th Ueda class}, see also \cite{N}. 

Serre's example \cite[p. 232--234 Example 3.2]{H} gives a typical example of class $(\alpha)$ with $X$ compact K\"ahler. See also \cite[Example 4.19]{PS}, \cite[Example 1.7]{DPS1}, and \cite[Example 3.5]{K20152} for this example. 
In this case we cannot apply Corollary \ref{cor:main2} even under the assumption that $X$ is compact K\"ahler, since it follows from \cite[Theorem 1]{U} and \cite[p. 594 Corollary]{U} that the image of the natural map $\textstyle\varinjlim_\mu \Gamma(X, \mathcal{O}_M(K_M)\otimes \mathcal{O}_M(\mu D))\to \Gamma(M, \mathcal{O}_M(K_M))$ {\it cannot} be dense, whereas $H_c^2(M, \mathcal{O}_M)$ is Hausdorff. 
Therefore it seems worth asking: 
\begin{question}
Does the FPLB hold on neighborhoods of $Y$ in $X$ when $(Y, X)$ is of class $(\alpha)$, or when $(Y, X)$ is as in Serre's example? 
\end{question}

For $(Y, X)$ of class $(\beta')$, it is shown by Ueda that $U|_{V_0}$ is holomorphically trivial for a neighborhood $V_0$ of $Y$ \cite[Theorem 3]{U}. 
In this case, 
it is known that $Y$ is a (maybe multiple) fiber of a fibration $p\colon X\to R$ onto a compact Riemann surface $R$ if $X$ is compact K\"ahler \cite[Theorem 5.1]{N}. 
As $p$ is proper, we have $\Gamma(p^{-1}(\Delta), \mathcal{O}_X) = p^*(\Gamma(\Delta, \mathcal{O}_R))$ for a neighborhood $\Delta$ of the image of $Y$ by $p$. 
Therefore it follows from 
Lemma \ref{lem:well-def_obstretc} $(iii)$ and the same argument as in \cite[p. 588--589]{U} that, for any holomorphic line bundle $L$ on a neighborhood $V$ of $Y$, the $\mu$-th obstruction class $u_\mu(\mathcal{I}, L)$ is well-defined whenever $(i_{\mu-1}^{\mathcal{I}_Y})^*\mathcal{O}_V(L)$ is trivial ($\mu\geq 1$). 
\begin{question}
Does the FPLB hold on neighborhoods of $Y$ in $X$ when $(Y, X)$ is of class $(\beta')$, or when $Y$ is a fiber of a fibration $p\colon X\to R$ from a compact K\"ahler surface onto a compact Riemann surface $R$? 
\end{question}

For class $(\beta'')$, some sufficient conditions for the FPLB have been given in terms of irrational theoretical properties of the normal bundle, refer to \cite[\S 3]{KUprojK3} and \cite[Lemma 4.6]{K20242} for example. 
The following gives a typical example of $(Y, X)$ with $X$ compact K\"ahler. 
\begin{example}\label{eg:toroidal_ruled}
Let $C$ be a compact Riemann surface and $F$ be a non-torsion element of ${\rm Pic}^0(C)$. 
Denote by $X$ the ruled surface which is obtaind by the fiberwise projectivization of the vector bundle $E:=\mathbb{I}_C\oplus F$, where $\mathbb{I}_C$ denotes the holomorphically trivial line bundle on $C$. 
Let $Y$ be the section of the projection $p\colon X\to C$ which corresponds to the subbundle $F\subset E$. 
Then $X$ can naturally be regarded as a compactification of the total space of $F$ which has $Y$ as the zero section. 
Denote by $V$ the total space of $F$ which is regarded as a subset of $X$. 
Then one can naturally identify $[Y]|_V$ with $(p|_V)^*F$, which is unitary flat since $F$ is unitary flat by a theorem of Kashiwara. Thus $(Y, X)$ is of class $(\beta'')$ in this example. 
\end{example}

When $N_{Y/X}$ is {\it non-torsion} (i.e., it is a non-torsion element of ${\rm Pic}^0(Y)$), the FPLB problem for the Ueda line bundle $U$ occurs when one try to decide whether $(Y, X)$ is of class $(\beta'')$. 
Whereas Ueda explicitly constructed an example of class $(\gamma)$ in \cite[\S 5.4]{U} and there is a slight generalization of his construction \cite{KO}, it seems that no example is known for $(Y, X)$ of class $(\gamma)$ when $X$ is compact. 
In order to investigate whether or not there exists an example of class $(\gamma)$ with $X$ compact, studying the following example should be important. 
\begin{example}\label{ex:blP29pts}
Let $C_0$ be a non-singular cubic curve in the projective plane $\mathbb{P}^2$, 
$\pi_1\colon X_1 \to \mathbb{P}^2$ be the blow-up at a point of $C_0$, 
and $\pi_\nu\colon X_\nu \to X_{\nu-1}$ be the blow-up of $X_{\nu-1}$ at a point of the strict transform $C_{\nu-1}$ of $C_{\nu-2}$ for $\nu=2, 3, \dots, 9$. 
Denote $X_9$ by $X$, the strict transform of $C_8$ by $Y$, and 
by $M$ the complement $X\setminus Y$. 
It follows from the blow-up formula that $K_X\cong [-Y]$ holds and that, via the pull-back by the restriction of $\pi:=\pi_1\circ\pi_2\circ\cdots\circ\pi_9$, the normal bundle $N_{Y/X}$ corresponds to the line bundle $[C_0]|_{C_0}\otimes[-q_1-q_2\cdots -q_9]$ on $C_0$, where $q_\nu$ is the image of the center of the $\nu$-th blow-up. 
As ${\rm deg}([C_0]|_{C_0}\otimes[-q_1-q_2\cdots -q_9])=3^2-9=0$, we have $N_{Y/X}\in {\rm Pic}^0(Y)$. 
It is classically well-known that $X$ admits a structure as an elliptic surface if and only if $N_{Y/X}$ is torsion, and $Y$ is a fiber of the elliptic fibration in this case (See \cite[\S 5]{N} for example). 
From Arnol'd's and Ueda's theorems \cite{A} \cite{U} it follows that $(Y, X)$ is of class $(\beta'')$ if $N_{Y/X}$ satisfies {\it the Diophantine condition}: i.e., if there exist positive numbers $A$ and $\alpha$ such that $d(\mathbb{I}_Y, N_{Y/X}^{\mu})\geq A\cdot n^{-\mu}$ holds for any positive integer $\mu$, where $d$ is an invariant distance on ${\rm Pic}^0(Y)$ which is defined in \cite[\S 4.5]{U} (Note that this distance is Lipschitz equivalent to the Euclidean distance on ${\rm Pic}^0(Y)$, see \cite[\S A.3]{KUK3}). 
Moreover, it is also shown in \cite[\S 3]{KUprojK3} that the FPLB holds on neighborhoods of $Y$ in $X$ when $N_{Y/X}$ satisfies the Diophantine condition. 
When $N_{Y/X}$ is non-torsion, it is shown by \cite{Br} (and generalized in \cite[Corollary 1.5]{K2024}) that $(Y, X)$ is of class $(\beta'')$ if and only if $K_X^{-1} (\cong [Y])$ is {\it semi-positive}: i.e., it admits a $C^\infty$ Hermitian metric whose Chern curvature form is positive semi-definite at any point of $X$. 
See also \cite{DPS}, \cite[\S 1]{D}, and \cite{K2019} for this example. 
\end{example}

In the rest of this section, we show the following for the FPLB problem on this example: 
\begin{proposition}\label{prop:cor3}
Let $X, Y$, and $M$ be as in Example \ref{ex:blP29pts}. 
Assume that $N_{Y/X}$ is non-torsion. 
Then the following are equivalent: \\
$(i)$ The FPLB holds on neighborhoods of $Y$ in $X$. \\
$(ii)$ $K_X^{-1}$ is semi-positive and $H_c^{0, 2}(M)$ is Hausdorff. \\
$(iii)$ $K_X^{-1}$ is semi-positive and $H^{2, 1}(M)=0$. \\
$(iv)$ $K_X^{-1}$ is semi-positive and $H^{2, 1}(M)$ is finite dimensional. 
\end{proposition}
 
In the proof of Proposition \ref{prop:cor3}, we apply \cite[Theorem 1.1]{K20242}, in which the cohomology $H^{2, 1}(M) (\cong H^{0, 1}(M))$ is investigated for Example \ref{ex:blP29pts} under the assumption that $K_X^{-1}$ is semi-positive. 
As the topology of $H^{2, 1}(M)$ that we are using in the present paper is different from that in \cite{K20242}, we first show the following as a preparation for the proof of the proposition: 
\begin{lemma}\label{lem:onhausdorffH1}
Let $X, Y$, and $M$ be as in Example \ref{ex:blP29pts}. 
Assume that $N_{Y/X}$ is non-torsion and $K_X^{-1}$ is semi-positive. 
Then the following are equivalent: \\
$(i)$ $H_c^{0, 2}(M)$ is Hausdorff. \\
$(ii)$ $H^{2, 1}(M)$ is Hausdorff in the topology as in \cite[\S 1.4.3]{Obook}: i.e., in the topology given by considering the resolution of $\mathcal{O}_M(K_M)$ by the sheaves of currents and giving the Fr\'echet topology of the uniform convergence on bounded sets (the strong dual topology) on the spaces of currents. \\
$(iii)$ $H^{2, 1}(M)$ is Hausdorff in the topology as in \cite{K20242}: i.e., in the topology given by considering the resolution of $\mathcal{O}_M(K_M)$ by the sheaves of $C^\infty$ forms and giving the Fr\'echet topology of uniform convergence on compact sets in all the derivatives of coefficient functions on the spaces of $C^\infty$ forms. \\
$(iv)$ $H^{2, 1}(M)=0$. \\
$(v)$ $H^{2, 1}(M)$ is finite dimensional. 
\end{lemma}

\begin{proof}
The implication $(iii)\implies (iv)$ is a direct consequence of \cite[Theorem 1.1]{K2024} (Note that $H^{2, 1}(M) = H^1(M, \mathcal{O}_M(K_M)) =  H^1(M, \mathcal{O}_M) =H^{0, 1}(M)$ follows because $K_M=K_X|_M\cong [-Y]|_M$ is holomorphically trivial, and that $H^{0, 1}(X)=0$ because $X$ is a rational surface). 
The implication $(iv)\implies (v)$ is clear, and $(v)\implies (ii)$ follows from Banach's open mapping theorem, refer to \cite[Lemme 2]{Se}. 
By virtue of a theorem of Larent-Thi\'{e}baut--Leiterer \cite{LL1} (refer also to \cite[Theorem 1.33]{Obook}), the conditions $(i)$ and $(ii)$ are equivalent to each other. 
Therefore it is enough to show the implication $(ii)\implies (iii)$ (Although it is just a part of \cite[Theorem 2.1]{La}, here we describe a proof of this implication for the reader's convenience). 
Denote by $K_{(ii)}$ the topological vector space defined by topologizing $H^{2, 1}(M)$ as in \cite[\S 1.4.3]{Obook}, by $K_{(iii)}$ the topological vector space defined by topologizing $H^{2, 1}(M)$ as in \cite{K20242}, and by $f\colon K_{(iii)}\to K_{(ii)}$ the linear map obtained by considering the identity map on $H^{2, 1}(M)$. 
Then $f$ is the map induced from the natural map $\iota\colon \Gamma(M, \mathcal{A}_M^{2, 1}) \to (\Gamma_c(M, \mathcal{A}_M^{0, 1}))^\vee$. 
This map $\iota$ is continuous, 
since, on any compact subsets, the integration of the limit coincides with the limit of the integrations for any volume forms which uniformly converge on a compact subset. 
Therefor $f$ is also continuous, which proves the implication $(ii)\implies (iii)$ because any topological space which admits an injective continuous map to a Hausdorff space is Hausdorff. 
\end{proof}

We also use the following:  
\begin{lemma}\label{lem:constr_Leviflat_holconst}
Let $\Omega$ be a complex manifold and $Y$ be a connected compact non-singular hypersurface of $\Omega$. 
Assume that $M:=\Omega\setminus Y$ is connected, $N_{Y/\Omega}$ is non-torsion element of ${\rm Pic}^0(Y)$, and that there exists a neighborhood $V$ of $Y$ in $\Omega$ such that $[Y]|_V$ is unitary flat. 
Then $M$ admits no non-constant holomorphic function. 
\end{lemma}

\begin{proof}
Let $\{H_m\}_m$ be a family of compact Levi-flat hypersurfaces on a neighborhood of $Y$ as in the proof of Proposition \ref{prop:variantofo2007}. 
For $m>>1$ fixed, any leaf of the Levi foliation of $H_m$ is dense, since the holonomy of this foliation coincides with the monodromy of the unitary flat line bundle $N_{Y/\Omega}$. 
Thus the assertion can be shown by applying the standard argument (See also the arguments in the proofs of \cite[Lemma 2.2]{KUK3} and \cite[Proposition 2.8, Lemma 2.10]{K20242}) as follows: 
For $f\in \Gamma(M, \mathcal{O}_M)$, the maximum principle applied to $f|_{\mathcal{L}}$ implies that $f$ is constant on $\overline{\mathcal{L}}=H_m$, where $\mathcal{L}$ is the leaf of the Levi foliation on $H_m$ which includes a point at which the maximum value of $|f|_{H_m}|$ is attained. 
Thus, for some constant $c$, $\{x\in M\mid f(x) = c\}$ is an analytic subset which includes a real hypersurface $H_m$, which implies that $f\equiv c$ holds on $M$. 
\end{proof}

\begin{proof}[Proof of Proposition \ref{prop:cor3}]
The equivalence of the assertions $(ii), (iii)$, and $(iv)$ is a direct consequence of Lemma \ref{lem:onhausdorffH1}. 
To prove the implication $(ii) \implies (i)$, assume that $K_X^{-1}$ is semi-positive and $H_c^{0, 2}(M)$ is Hausdorff. 
From \cite{Br} we have that $K_X^{-1}|_V=[Y]|_V$ is flat for a neighborhood $V$ of $Y$ in $X$ (see also \cite[Corollary 1.5]{K2024}). 
Thus, from Lemma \ref{lem:constr_Leviflat_holconst} we infer that $M$ has no non-constant holomorphic function. 
Therefore $\Gamma(M, \mathcal{O}_M(K_M))=\mathbb{C}\cdot \sigma$, where $\sigma$ is a meromorphic $2$-form on $M$ which generates $\Gamma(X, \mathcal{O}_X(K_X)\otimes \mathcal{O}_X(Y))$ (Recall that $K_X\cong [-Y]$ follows from the blow-up formula). 
Thus the assertion $(i)$ follows from $(ii)$ by virtue of Corollary \ref{cor:main2}. 
In what follow we show (the contraposition of) the implication $(i) \implies (ii)$. Assuming that the assertion $(ii)$ does not hold, we will show that the FPLB does not hold on neighborhoods of $Y$ in $X$. 

First let us consider the case where $K_X^{-1}$ is not semi-positive. 
In this case, it follows from \cite[Corollary 1.5]{K2024} that $[Y]|_V$ is not unitary flat for any neighborhood $V$ of $Y$. 
Therefore the Ueda line bundle $U:=[Y]|_V\otimes \widetilde{N_{Y/X}}^{-1}$ is not trivial for any $V$. 
By an inductive application of Lemma \ref{lem:well-def_obstretc} $(i)$, we have that $i_\infty^*\mathcal{O}_V(U)$ is trivial, since $U|_Y$ is trivial by construction and $H^1(Y, \mathcal{I}_Y^\mu/\mathcal{I}_Y^{\mu+1})=H^1(Y, N_{Y/X}^{-\mu})=0$ holds for any $\mu\geq 1$ because $Y$ is an elliptic curve and $N_{Y/X}$ is non-torsion (Recall Remark \ref{rmk:van_of_N-impliesj^*triv}. Here we also used the basic fact that $H^1(Y, \mathcal{O}_Y(F))=0$ holds for any holomorphically non-trivial unitary flat line bundle $F$ on an elliptic curve $Y$, see \cite[\S 1.2]{U} for example). 
Thus the FPLB does not hold on neighborhoods $V$ of $Y$ in $X$ in this case. 

Next we show the assertion by assuming that $K_X^{-1}$ is semi-positive. In this case $H_c^{0, 2}(M)$ is non-Hausdorff, since we are assuming that the assertion $(ii)$ does not hold. 
Consider the canonically defined continuous linear map 
$\iota^{2, 0}\colon H_c^{0, 2}(M) \to (H^{2, 0}(M))^\vee$, where the topology of $H^{2, 0}(M)$ is as in \cite[\S 1.4.3]{Obook}. 
We have that $(H^{2, 0}(M))^\vee$ is topologically isomorphic to $\mathbb{C}$ with Euclidean topology, since $H^{2, 0}(M)=\Gamma(M, \mathcal{O}_M(K_M))=\mathbb{C}\cdot \sigma$ holds again by Lemma \ref{lem:constr_Leviflat_holconst}, where $\sigma$ is a meromorphic $2$-form on $M$ as above. 
It is known that $\iota^{2, 0}$ is surjective (see \cite[\S 1.4.3]{Obook}). 
Therefore we have that ${\rm dim}\,H_c^{0, 2}(M)>1$ holds, since otherwise $\iota^{2, 0}$ is an injective continuous map to a Hausdorff space, which contradicts to the non-Hausdorffness of $H_c^{0, 2}(M)$. 
Thus the assertion follows from Lemma \ref{lem:forprfcor3} below, since $H^2(X, \mathcal{I}_Y)=H^2(X, \mathcal{O}_X(-Y))=H^2(X, \mathcal{O}_X(K_X))=H^{2, 2}(X)\cong \mathbb{C}$. 
\end{proof}

\begin{lemma}\label{lem:forprfcor3}
Let $X$ be a connected compact K\"ahler surface, 
$\mathcal{I}\subset \mathcal{O}_X$ be a coherent ideal sheaf, 
$Y$ be the support of $\mathcal{O}_X/\mathcal{I}$, and $M$ be the complement $X\setminus Y$. 
Assume that $Y\not=\emptyset$ and that ${\rm dim}\,H_c^{0, 2}(M)>{\rm dim}\,H^2(X, \mathcal{I})$ holds. 
Then there exist a neighborhood $V$ of $Y$ and a holomorphic line bundle $L$ on $V$ such that $(i_0^{\mathcal{I}})^*\mathcal{O}_V(L)$ is trivial and $L|_{V_0}$ is holomorphically non-trivial for any neighborhood $V_0$ of $Y$ in $V$. 
In particular, the FPLB does not hold on neighborhoods of $Y$ in $X$ if $H^1(Y, \mathcal{I}^\mu/\mathcal{I}^{\mu+1})=0$ holds for any $\mu\geq 1$. 
\end{lemma}

\begin{proof}
Denote by $N$ the dimension of $H^2(X, \mathcal{I})$. 
From the assumption, one can take $N+1$ elements $v_0, v_1, \dots, v_N\in H_c^{0, 2}(M)$ which are linearly independent. 
As $M$ is a connected non-compact surface, it follows that 
$H^{0, 2}(M) = 0$ \cite{M} (see also \cite[\S 1]{D1}). 
Thus we infer from the exactness of $(\ref{exseq:long_ex_seq_H_cHlimH})$ that the natural map $\textstyle\varinjlim_K H^{0, 1}(M\setminus K)
\to H_c^{0, 2}(M)$ is surjective. 
Therefore one can take an element $\eta_\nu\in\textstyle\varinjlim_V H^{0, 1}(V\setminus Y) (=\textstyle\varinjlim_K H^{0, 1}(M\setminus K))$ which is mapped to $v_\nu$ by this map ($\nu=0, 1, \dots, N$). 
As ${\rm dim}\,H^2(X, \mathcal{I})=N$, there exists an element $(c_0, c_1, \dots, c_N)\in \mathbb{C}^{N+1}\setminus \{(0, 0, \dots, 0)\}$ such that $\eta := \textstyle\sum_{\nu=0}^N c_\nu\eta_\nu$ is mapped to zero by the map $\textstyle\varinjlim_V H^{0, 1}(V\setminus Y)\to H^2(X, \mathcal{I})$ which appears in the sequence
\[
%H^1(X, \mathcal{I})\to 
H^{0, 1}(M)\oplus \varinjlim_V H^1(V, \mathcal{I})
\to \varinjlim_V H^{0, 1}(V\setminus Y)
\to H^2(X, \mathcal{I})
\]
obtained by the same argument as in the proof of Proposition \ref{prop:simple_thmmaini}. 
From the exactness of this sequence, we infer that there exists an element $\zeta \in H^{0, 1}(M)$ and $\xi\in \textstyle\varinjlim_V H^1(V, \mathcal{I})$ such that (the restriction of) $-\zeta+\xi$ coincides with $\eta$. By construction it follows that the image of $\xi$ by the composition of the natural maps 
\[
\varinjlim_V H^1(V, \mathcal{I}) \to \varinjlim_V H^{0, 1}(V\setminus Y)\to H_c^{0, 2}(M)
\]
is $v := \textstyle\sum_{\nu=0}^N c_\nu v_\nu$. 
Take a holomorphic line bundle $L$ on a neighborhood of $Y$ which represents the element $\alpha_{\mathcal{I}, 1}(\xi)\in \textstyle\varinjlim_V H^1(V, \mathcal{O}_V^*)$: i.e., $[L]_Y=\alpha_{\mathcal{I}, 1}(\xi)$. 
As $v(\mathcal{I}, L)=v$ holds by construction and $v\not=0$ holds by the linear independence of $v_\nu$'s, the former half of the assertion follows from Lemma \ref{lem:iikae_Ltriv_fund} $(i)$ and 
Proposition \ref{prop:van_v_implies_Ltriv}. 
The latter half of the assertion follows from the former half by an inductive application of Lemma \ref{lem:well-def_obstretc} $(i)$ (see Remark \ref{rmk:van_of_N-impliesj^*triv}). 
\end{proof}

It should to be natural to ask the following, since some parts of the argument in the proof of Proposition \ref{prop:cor3} seem to be also applicable to $(Y, X)$ as in 
Example \ref{eg:toroidal_ruled} (See also \cite[\S 1, \S 7.1]{K20242}). 
\begin{question}
Can one show an analogue of Proposition \ref{prop:cor3} for $(Y, X)$ as in 
Example \ref{eg:toroidal_ruled}? 
\end{question}

\begin{proof}[{Proof of Theorem \ref{thm:exp2bl9pts}}]
Let $C_0\subset \mathbb{P}^2$ be a non-singular cubic curve, 
$\pi_1\colon X_1 \to \mathbb{P}^2$ be the blow-up at a point of $C_0$, 
and $\pi_\nu\colon X_\nu \to X_{\nu-1}$ be the blow-up of $X_{\nu-1}$ at a point of the strict transform $C_{\nu-1}$ of $C_{\nu-2}$ for $\nu=2, 3, \dots, 8$. Denote by $C_8$ the strict transform of $C_7$. 
Let $\mathcal{X}$ be the blow-up of $X_8\times C_8$ along the analytic subset 
$\{(t, t) \mid t\in C_8\}\subset X_8\times C_8$, 
$\mathcal{Y}$ the strict transform of $C_8\times C_8$, 
and $\pi\colon \mathcal{X}\to R:=C_8$ be the composition of the blow-up morphism and the second projection $X_8\times C_8\to C_8$. 
Note that $X_t:=\pi^{-1}(t)$ (resp. $Y_t:=X_t\cap \mathcal{Y}$) coincides with $X$ (resp. $Y$) in Example \ref{ex:blP29pts} when one choose $t\in C_8$ as the center of the ninth blow-up. 
By construction, $X_t$ and $Y_t$ clearly satisfy the condition $(i)$. 
For almost every $t\in R$, $N_{Y_t/X_t}$ satisfies the Diophantine condition (see \cite[Remark 1.4]{KUK3} for example). 
As $K_{X_t}^{-1}$ is semi-positive in this case (see Example \ref{ex:blP29pts}), the assertion $(ii)$ follows from \cite[Theorem 1.1]{K20242} and Proposition \ref{prop:cor3}. 
The assertion $(iii)$ follows from Proposition \ref{prop:cor3}, \cite[Corollary 1.3]{K20242}, and \cite[p. 155 Satz I]{C}. 
\end{proof}

%%%%%%%%%%%%%%%%%%%%%%%%%%%%

%%%%%%%%%%%% References %%%%%%%%%%%%%

%%%%%%%%%%%%%%%%%%%%%%%%%%%%%%%%

%%%%%%%%%%%%%%%%%%%%%%%%%%%%%%%%%%%%

\end{document}